\documentclass[12pt]{amsart}
\usepackage{amssymb}
\usepackage{tikz}
\newtheorem{theorem}{Theorem}[section]
\newtheorem{lemma}[theorem]{Lemma}
\newtheorem{corollary}[theorem]{Corollary}

\newtheorem{proposition}[theorem]{Proposition}
\newtheorem{remark}[theorem]{Remark}
\newtheorem{definition}[theorem]{Definition}

\newtheorem{example}[theorem]{Example}
\newcommand{\ncom}{\newcommand}
\ncom{\lrar}{\longrightarrow}
\ncom{\rar}{\rightarrow}
\ncom{\ov}{\overline}
\ncom{\m}{\mbox}
\ncom{\sta}{\stackrel}
\ncom{\comx}{{\mathbb C}}
\ncom{\A}{{\mathbb A}}
\ncom{\Z}{{\mathbb Z}}
\ncom{\Q}{{\mathbb Q}}
\ncom{\R}{{\mathbb R}}
\ncom{\al}{\alpha}
\ncom{\p}{{\mathbb P}}
\ncom{\N}{{\mathbb N}}
\ncom{\f}{\frac}
\ncom{\cD}{{\mathcal D}}
\ncom{\cO}{{\mathcal O}}
\ncom{\cP}{{\mathcal P}}
\ncom{\cV}{{\mathcal V}}
\ncom{\cF}{{\mathcal F}}
\ncom{\what}{\widehat}
\ncom{\delbar}{\overline{\partial}}
\ncom{\eop}{{\hfill $\Box$}}

\ncom{\Wone}{W^{(1)}}
\ncom{\Wtwo}{W^{(2)}}
\ncom{\Lex}{\mathrm{Lex}}
\ncom{\Grr}{\mathrm{Gr}}

\begin{document}
\baselineskip=16pt
\author[J.\,N.\,Iyer]{Jaya NN Iyer}
\address{Institute for Mathematical Sciences\\ Chennai, India}
\email{jniyer@imsc.res.in}
\address{Max-PLanck Institute for Mathematics, Bonn, Germany}
\email{jniyer@mpim-bonn.mpg.de}
\author[C.\,Simpson]{Carlos Simpson}
\address{CNRS, Laboratoire J.\,A.\,Dieudonn\'{e}, UMR 6621\\
	Universit\'{e}  C\"ote d'Azur,\\
	06108 Nice, Cedex 2, France}
\email{carlos.simpson@univ-cotedazur.fr}
\title[Regulators of canonical extensions: two intersecting divisors]
{Regulators of canonical extensions are torsion:\\ the case of two transversally intersecting smooth divisors}

\begin{abstract}
This note extends the main result of \cite{IS} --- torsion of the extended
Chern--Simons (regulator) classes of the Deligne canonical extension of a flat
bundle with unipotent monodromy at infinity --- from the case of a smooth
irreducible boundary divisor to the case of a boundary divisor
$D = D_1\cup D_2$ with two smooth irreducible components meeting transversally
along a smooth center $Z=D_1\cap D_2$. Let $X$ be a smooth projective variety  defined over $\mathbb{C}$, and $U:=X-D$. Given a flat bundle $(E,\nabla)$ on $U$ with unipotent monodromy around the components of $D$ consider Deligne's canonical extension $(\ov{E},\bar{\nabla})$ on $X$. Then the extended Chern-Simons classes
$$
c_p(\ov{E},\ov{\nabla})\in H^{2p-1}(X,\mathbb{C}/\mathbb{Z})
$$
are torsion, for $p\geq 2$. These notes were prepared in 2009-2010, and the preprint \cite[2026]{IS2} treats the full normal crossing case via a different approach.
\end{abstract}
\maketitle

\setcounter{tocdepth}{1}
\tableofcontents

\section{Introduction and statement of the result}

Throughout, $X$ is a smooth quasi--projective variety over $\comx$ and
$$
D \;=\; D_1\cup D_2 \;\subset\; X
$$
is a divisor with exactly two irreducible components, each smooth, meeting
transversally along the smooth (not necessarily connected) center
$$
Z \;:=\; D_1\cap D_2 .
$$
Put $U:= X-D$. Let $\gamma_1,\gamma_2$ denote loops in $U$ going once around
$D_1$, resp.\ $D_2$ (based suitably; near a point of $Z$ both loops are
available and commute). Let
$$
\rho : \pi_1(U)\lrar GL_r(\comx)
$$
be a representation which is \emph{unipotent at infinity}: $\rho(\gamma_1)$
and $\rho(\gamma_2)$ are unipotent. Let $(E,\nabla)$ be the associated flat
bundle on $U$, $L = E^{\nabla}$ the local system, and $(\ov E,\ov\nabla)$ the
Deligne canonical extension \cite{De} across $D$, a vector bundle on $X$ with
logarithmic connection having nilpotent residues along $D_1$ and $D_2$.

The purpose of this note is to extend the constructions and the main theorem
of \cite{IS} to this situation:

\begin{theorem}
\label{mainthm}
In the above situation there are canonically defined extended regulator
classes
$$
\what{c}_p(\rho / X)\,\in\, H^{2p-1}(X,\comx/\Z), \qquad p\geq 1,
$$
restricting to the usual Chern--Simons classes of $(E,\nabla)$ on $U$,
additive in $\rho$, contravariantly functorial in $(X,D)$, and rigid under
deformations of $\rho$ which stay unipotent at infinity. For $p\geq 2$ these
classes are torsion. If moreover $X$ is projective, the classes
$\what{c}_p(\rho/X)$ lift the Deligne Chern classes
$c_p^{\cD}(\ov E)\in H^{2p}_{\cD}(X,\Z(p))$, which are therefore torsion for
$p\geq 2$.
\end{theorem}

The two difficulties pointed out in
\cite{IS} for the normal--crossings case (failure of the common--refinement
Condition 4.6 of \cite{IS} at the corner, and the failure of the curvature
argument of \cite[Remark 6.9]{IS} for a two--parameter deformation collar
$S\times [0,1]^2$) are resolved here as follows. First, the patching formalism
is generalized to a \emph{bifiltered} patching formalism: at the corner one
keeps the \emph{pair} of commuting monodromy filtrations instead of trying to
refine one into the other, using the elementary fact that \emph{two}
filtrations of a vector bundle always admit a simultaneous (bi-graded)
splitting. Second, the comparison with $K$-theory is carried out using a
\emph{cubical} homotopy pushout whose corner carries the two--variable
deformation space $BGL(F[t_1,t_2])^+$; the two--parameter deformation is a
conjugation (``triangular'') deformation, so the curvature of the resulting
deformed connection is strictly triangular for the corner bifiltration, and
all Chern forms vanish identically --- no codimension argument is needed.
The hermitian counterpart, with $O_{p,q}(\comx[t_1,t_2])$ and Karoubi's
homotopy invariance applied twice, together with an elementary self--dual
splitting of the pair of weight filtrations (for which the several--variable
$SL_2$--orbit theorems of Cattani--Kaplan--Schmid provide refined
representatives), yields
vanishing of the extended volume regulators for complex variations of Hodge
structure, and then Mochizuki's theorem and Reznikov's argument give torsion.

\subsection{}\label{obstacles} Two specific obstructions to the normal--crossings case were recorded in
\cite{IS}.

\textbf{(A) Failure of the common refinement condition.}
The patching formalism of \cite[\S 4]{IS} requires
(\cite[Condition 4.6]{IS}) that on the overlap of two patching neighborhoods
the two filtrations admit a \emph{common refinement}. Near a corner point of
$Z$ the natural filtrations are the kernel (or monodromy weight) filtrations
$\Wone$ of $N_1=\log\rho(\gamma_1)$ and $\Wtwo$ of $N_2=\log\rho(\gamma_2)$.
Two filtrations of a vector space admit a common refinement if and only if
their members are pairwise comparable under inclusion (\cite[Lemma 4.4]{IS}),
and this essentially never holds for $\Wone,\Wtwo$.

The remedy is not to refine but to keep the pair. The relevant linear--algebra
fact is:
\emph{any two filtrations of a vector space (or, locally, of a
$\mathcal{C}^\infty$ vector bundle) admit a simultaneous splitting, i.e.\ a
bigrading $V=\bigoplus_{a,b}V_{a,b}$ with
$\Wone_a=\bigoplus_{a'\leq a}V_{a',\bullet}$ and
$\Wtwo_b=\bigoplus_{b'\leq b}V_{\bullet,b'}$} (Deligne's two--filtration
lemma; for three or more filtrations this fails, which is precisely the
combinatorial shadow of triple points of a normal--crossings divisor).
We therefore generalize the notion of patching collection so that the corner
neighborhood carries the \emph{bifiltration} $(\Wone,\Wtwo)$ together with a
flat connection on the double associated--graded
$\Grr^{\Wtwo}\Grr^{\Wone}(E)$, and we require the patched connection to
preserve both filtrations there. The curvature of such a connection strictly
decreases the \emph{total degree} filtration
$F_c=\sum_{a+b\leq c}\Wone_a\cap\Wtwo_b$, so all Chern forms vanish
identically and the Cheeger--Simons differential character lands in
$H^{2p-1}(X,\comx/\Z)$ exactly as in \cite[\S 4.1]{IS}. This is carried out in
\S\S \ref{bifiltered}--\ref{construction}.

\textbf{(B) Failure of the codimension argument on the corner collar.}
In \cite[Remark 6.9]{IS} it is observed that on a two--parameter collar
$S\times[0,1]^2$ a connection which is flat on each slice
$S\times\{(t_1,t_2)\}$ no longer satisfies $\Omega\wedge\Omega=0$: the
curvature may contain $dt_1\wedge dt_2$--terms, so the vanishing of the higher
Chern forms cannot be deduced from a codimension count.

The remedy is that for the specific deformation we use --- conjugation by the
one--parameter subgroups attached to a simultaneous bigrading of
$(\Wone,\Wtwo)$ --- the family of representations over $[0,1]^2$ preserves a
\emph{fixed} flag pair, and its block--diagonal part is \emph{constant} in
$(t_1,t_2)$. Hence the deformed connection $\nabla^{\rm def}$ on
$S\times[0,1]^2$ preserves the (constant) bifiltration and induces a constant
flat connection on the double graded; its curvature, $dt_i$--terms included,
is strictly triangular, and \emph{all} Chern forms vanish identically. The
codimension argument is thereby replaced by a triangularity argument, valid
for every $p\geq 1$. This is carried out in \S \ref{kdef}, where the
relevant universal space is the cubical homotopy pushout built from
$BGL(F[t_1,t_2])^+$.


Both new ingredients of this note are specific to corners of multiplicity
two, and both fail at a triple point $D_1\cap D_2\cap D_3$:
\begin{itemize}
\item[(i)] three filtrations of a vector space do \emph{not} in general admit
a simultaneous splitting; the construction of the corner model connection
$\nabla_{123}$ preserving $W^{(1)},W^{(2)},W^{(3)}$ with flat triple graded
breaks down, and with it the pointwise triangularity of the patched
connection on overlaps of corner patches of different multiplicities.
Mochizuki's constancy of the weight filtrations $W(\sum a_iN_i)$, $a_i>0$,
\cite{Mochizuki} suggests using the single filtration $W(N_1+\cdots+N_k)$ on
the deepest stratum, but the required compatibilities with the lower--depth
filtrations on overlaps are exactly the open problem already noted in
\cite[\S 4.4]{IS};
\item[(ii)] even granting the splittings, the triple
deformation $\m{Ad}(\psi^1_{t_1}\psi^2_{t_2}\psi^3_{t_3})(\rho)$ remains
triangular, so in fact Lemma \ref{triangulardef} would survive; the genuinely
open point in higher multiplicity is (i), i.e.\ the simultaneous splitting
and the Jordan--H\"older string argument for tuples of filtrations needed for
Proposition \ref{biindepfilt}. We record this as the precise remaining
obstruction for the general normal--crossings case.
\end{itemize}
For a normal crossings divisor whose components are pairwise disjoint or, more
generally, such that at most two branches pass through any point, the
arguments of this note apply without change.

\textit{The general normal--crossings case: overview of \cite{IS2}}
\label{generalNCsummary}

 The article \cite{IS2} treats an arbitrary normal--crossings divisor
$D=\bigcup_i D_i$. We summarise it here to place the present note in
context and because the two constructions are compared in detail in
\S \ref{reescomparison}.

Given a local system on $X-D$ with unipotent monodromy, defined over a
number field $K$, \cite{IS2} proceeds in five steps. (1) A \emph{good
adapted covering} $X=\bigcup_I U_I$, indexed by multi--indices $I$ of
divisor components, is constructed so that non--empty multiple intersections
are linearly ordered by inclusion of index sets and distinct index classes
do not cross --- the combinatorial replacement for our single triple
$(V_1,V_2,V_{12})$. (2) The canonical extension is equipped with a strict
collection of patching data $(W(I),\tau(I))$, built from the monodromy
weight filtrations $W(N_I)$, $N_I=\sum_{i\in I}N_i$, whose \emph{sequential}
compatibility along chains $I\subset J$ --- as opposed to the simultaneous
splitting of Definition \ref{bipatchdef}(II), which is exactly what fails at
a triple point --- is supplied by Mochizuki's theorem \cite{Mochizuki} on
weight filtrations of tame harmonic bundles. (3) An iterated multi--Rees
construction builds an algebraic vector bundle $F_K$ on the projective
cubical realization of the nerve of the covering, and the extended
regulator is identified with the Deligne--Beilinson Chern class of $F_K$ by
two independent comparisons: a purely $K$--theoretic one, resting
on homotopy invariance of $BGL(F[t_1,\dots,t_n])^+$ and requiring no
connection at all, and an analytic one, via the Dupont--Hain--Zucker
independence theorem and the arithmetic Chern character of Burgos--Gil
\cite{BurgosGil}, identifying the Deligne class with a Chern--Simons class.
 As here, vanishing of the volume regulator for a representation
underlying a variation of Hodge structure --- now via contractibility of a
poset of $N$--isotropic filtrations, proved with Quillen's Theorem A ---
combines with Borel's theorem \cite{Borel} and Reznikov's argument
\cite{Re2} to give torsion in degree $\geq 2$.

\subsection{Outline} \S \ref{geometry} fixes the cubical geometry of
neighborhoods. \S \ref{bifiltered} develops bifiltered patching collections
and the vanishing/invariance lemmas. \S \ref{construction} constructs the
patched connection for $(\rho,\Wone,\Wtwo)$ and proves well--definedness,
independence of the filtrations, additivity, functoriality and rigidity.
\S \ref{deligne} proves compatibility with the Deligne Chern class for $X$
projective. \S \ref{kdef} constructs the cubical deformation model
$BGL(F)^+_{{\rm def},2}$ using $BGL(F[t_1,t_2])^+$ and identifies the
resulting regulator classes with the patched ones. \S \ref{hermitian}
develops the hermitian counterpart over $\comx[t_1,t_2]$ and proves the
vanishing of the extended volume regulators for variations of Hodge
structure, using a self--dual splitting of the pair of weight filtrations
(Proposition \ref{CKSsplitting}, an elementary statement for which the
several--variable $SL_2$--orbit theorem of Cattani--Kaplan--Schmid
\cite{CKS} supplies distinguished representatives). \S \ref{proofmain} assembles the proof of
Theorem \ref{mainthm}. 

Throughout we freely use the notation, conventions
($\Z(p)$, differential characters $\what{H^k}(X,\comx/\Z)$, the exact
sequence (2.2) of \cite{IS}, the variational formula
\cite[Prop.\ 2.9]{Ch-Si}) and the results of \cite{IS}, which we refer to as
``the smooth--divisor case''.

\textbf{Acknowledgements} These notes were prepared in 2009-2010, extending \cite{IS}. The preprint \cite{IS2} treats the full normal crossing divisor case through a different route. The first named author thanks Max-Planck Institute for Mathematics, Bonn for their hospitality and support. She was partly supported by the IMSc-DAE Apex grant "Complex Algebraic Geometry". The second author was partly supported by the ERC Horizon Synergy grant 101167526 (MALINCA), the Horizon 2020 grant 670624 (Mai Gehrke's DuaLL project), and the ANR program IsoMoDyn (ANR-25-CE40-1360).

\textit{AI disclosure:} Proof-checking of the article used Claude AI.

\section{The cubical geometry of neighborhoods}
\label{geometry}

\subsection{Compatible tubular data}
Choose hermitian metrics on the normal bundles and compatible tubular
neighborhood structures so that near $Z$ the two tubular projections commute
(this is the standard ``corner'' structure for a pair of transversally
intersecting submanifolds; it exists by the usual tubular neighborhood
theorem applied first to $Z$ and then relatively). Concretely we fix:

\begin{itemize}
\item a \emph{polydisc bundle neighborhood} $B_{12}$ of $Z$: a locally trivial
fiber bundle $B_{12}\rar Z$ with fiber the bidisc $\Delta\times\Delta$, such
that $D_1\cap B_{12} = \{z_2$--coordinate $=0\}$ and
$D_2\cap B_{12}=\{z_1$--coordinate $=0\}$ in fiber coordinates
(so $D_i\cap B_{12}$ is a disc bundle over $Z$);
\item disjoint closed neighborhoods $K_1, K_2$ of $Z$ inside $B_{12}$,
$K_2\subset {\rm int}(K_1)$, both polydisc bundles of smaller polyradius;
\item tubular (disc bundle) neighborhoods $B_1$ of $D_1- {\rm int}(K_2)$ and
$B_2$ of $D_2-{\rm int}(K_2)$, chosen thin enough that
\begin{equation}\label{disjointness}
B_1\cap D_2=\emptyset,\qquad B_2\cap D_1 = \emptyset, \qquad
B_1\cap B_2\subset B_{12}.
\end{equation}
After shrinking $B_1,B_2$ further (replacing $D_i-{\rm int}(K_2)$ by
$D_i-{\rm int}(K_1)$ in the role of the core) we may and do also assume
\begin{equation}\label{disjointness2}
B_1\cap B_2 = \emptyset .
\end{equation}
\end{itemize}
Set $V_0:=U$, $V_1:=B_1$, $V_2:=B_2$, $V_{12}:=B_{12}$; this is an open cover
of $X$. Denote by a star the punctured pieces:
$V_i^{\ast}:=V_i - D$, $V_{12}^{\ast}:=V_{12}-D$.

\subsection{Fundamental groups at the boundary}
$V_{12}^{\ast}\rar Z$ is a $(\Delta^{\ast}\times\Delta^{\ast})$--bundle.
Hence there is an exact sequence
$$
\Z^2 \;=\;\langle \gamma_1,\gamma_2\rangle \lrar \pi_1(V_{12}^{\ast})
\lrar \pi_1(Z)\lrar 1 ,
$$
and since the structure group of each disc factor is the circle, the
conjugation action of $\pi_1(V_{12}^{\ast})$ on the image of $\Z^2$ is
trivial: \emph{the classes of $\gamma_1$ and $\gamma_2$ are central in
$\pi_1(V_{12}^{\ast})$, and they commute with each other.} Similarly
$\gamma_i$ is central in $\pi_1(V_i^{\ast})$. The maps
$\pi_1(V^{\ast})\rar\pi_1(V)$ are surjective for
$V=V_1,V_2,V_{12}$, with kernel normally generated by the relevant
$\gamma_i$'s (both of them for $V_{12}$).

Consequently, with $N_i:=\log\rho(\gamma_i)$ (well defined near the relevant
boundary stratum):
\begin{equation}\label{centrality}
\m{$\rho(\pi_1(V_i^{\ast}))$ commutes with $N_i$;}\quad
\m{$\rho(\pi_1(V_{12}^{\ast}))$ commutes with both $N_1$ and $N_2$;}\quad
[N_1,N_2]=0 .
\end{equation}

\subsection{The corner coordinates and the cubical decomposition}
\label{cubicaldecomp}
Let $r_i: X\rar [0,1]$ be smooth functions with $r_i^{-1}(0)=D_i$, equal to
the normalized tubular radius of $D_i$ on a neighborhood of $D_i$, and
$\equiv 1$ outside it; near $Z$ use the two polydisc radii. The map
$$
q := (r_1,r_2) : X \lrar [0,1]^2
$$
exhibits $X$ as glued from pieces indexed by the faces of the square
$\square=[0,1]^2$. Fix $0<\epsilon<\epsilon '<1$ and set
$$
Y_{11}:= q^{-1}([\epsilon ',1]^2), \quad
Y_{01}:= q^{-1}([0,\epsilon]\times[\epsilon ',1]),\quad
Y_{10}:= q^{-1}([\epsilon ',1]\times[0,\epsilon]),\quad
Y_{00}:= q^{-1}([0,\epsilon]^2),
$$
the four \emph{vertex blocks} (deep interior; collar of $D_1$ away from $Z$;
collar of $D_2$ away from $Z$; corner block around $Z$), and the
\emph{edge collars} and the \emph{corner collar}
$$
C \;\cong\; S_{12}\times [0,1]^2,
\qquad S_{12} := q^{-1}(\{(\epsilon ,\epsilon )\}) ,
$$
where $S_{12}$ is the $T^2$--bundle over $Z$ of ``corner directions''. (The
edge collars are of the form $S_1'\times[0,1]$ etc.\ where $S_1'$ is the part
of the circle bundle $\partial B_1$ lying over $D_1$ away from $Z$, and
similar pieces over the corner block.) The conclusion we use is purely
homotopy--theoretic:

\begin{lemma}\label{cubicalpushout}
$X$ is the homotopy colimit of the cubical diagram, indexed by the poset of
faces of the square $\square$, which assigns: to the four vertices the blocks
$Y_{11}\simeq U$, $Y_{01}\simeq D_1-Z$ (homotopy type of the open part of
$D_1$), $Y_{10}\simeq D_2-Z$, $Y_{00}\simeq Z$; to the four edges the
corresponding cylinders (homotopy types $S_1^{\circ}$, $S_2^{\circ}$,
$\partial$--pieces of the corner block); and to the $2$--face the corner
collar $C\simeq S_{12}$. All maps are the inclusions of boundary pieces.
\eop
\end{lemma}

This is the ``cubical homotopy pushout'' replacing the square pushout
diagram (1.2) of \cite{IS}; it is simply the statement that $X$ is built by
gluing the blocks along collars, exactly as in the smooth--divisor case but
with one $2$--dimensional gluing parameter at the corner.

\section{Bifiltered patching collections}
\label{bifiltered}

In this section $X$ is any smooth manifold and $E$ a $\mathcal{C}^{\infty}$
complex vector bundle on $X$. We extend \cite[\S 4.1]{IS} to allow pairs of
filtrations.

\begin{definition}\label{bipatchdef}
A \emph{$2$--patching collection} on $(X,E)$ consists of an open cover
$\{V_i\}_{i\in I}$ and, for each $i$, either
\begin{itemize}
\item[(I)] a single increasing filtration $W^i$ of $E|_{V_i}$ by strict
subbundles, together with a flat connection $\nabla_{i,\Grr}$ on
$\Grr^{W^i}(E|_{V_i})$ (a \emph{simple} patch, as in \cite{IS}); or
\item[(II)] an ordered pair of increasing filtrations
$(\Wone{}^{,i},\Wtwo{}^{,i})$ of $E|_{V_i}$ by strict subbundles such that all
intersections $\Wone_a\cap\Wtwo_b$ are strict subbundles (equivalently: the
pair admits, locally on $V_i$, a simultaneous splitting
$E=\bigoplus_{a,b}E_{a,b}$), together with a flat connection
$\nabla_{i,\Grr\Grr}$ on the double associated--graded
$$
\Grr\Grr(E|_{V_i}) \;:=\; \bigoplus_{a,b}
\frac{\Wone_a\cap\Wtwo_b}{\Wone_{a-1}\cap\Wtwo_b+\Wone_a\cap\Wtwo_{b-1}}
$$
(a \emph{corner} patch).
\end{itemize}
A connection $\nabla$ on $E$ is \emph{compatible} with the collection if over
each simple patch it preserves $W^i$ and induces $\nabla_{i,\Grr}$ on
$\Grr^{W^i}$, and over each corner patch it preserves \emph{both}
$\Wone{}^{,i}$ and $\Wtwo{}^{,i}$ and induces $\nabla_{i,\Grr\Grr}$ on
$\Grr\Grr$.
\end{definition}

Note that a connection preserving both filtrations of a corner patch
automatically preserves every $\Wone_a\cap\Wtwo_b$, hence acts on
$\Grr\Grr$.

Two total filtrations canonically attached to a corner pair will be used:
the \emph{total degree filtration}
\begin{equation}\label{totaldeg}
F_c \;:=\; \sum_{a+b\leq c} \Wone_a\cap\Wtwo_b ,
\end{equation}
which is a filtration by strict subbundles (use a local simultaneous
splitting to see local freeness), and the two \emph{lexicographic
refinements}
\begin{equation}\label{lexdef}
\Lex(\Wone,\Wtwo)_{(a,b)} \;:=\; \Wone_{a-1} + \big( \Wone_a\cap\Wtwo_b \big),
\qquad \m{ordered lexicographically in } (a,b),
\end{equation}
and symmetrically $\Lex(\Wtwo,\Wone)$. One checks, by the Zassenhaus
isomorphism, that the associated graded of $F$, of $\Lex(\Wone,\Wtwo)$ and of
$\Lex(\Wtwo,\Wone)$ are each canonically identified with $\Grr\Grr(E)$
(with different orderings of the summands). $\Lex(\Wone,\Wtwo)$ refines
$\Wone$, and is preserved by any endomorphism--valued object preserving both
$\Wone$ and $\Wtwo$, and also by any connection which preserves $\Wone$ and
whose induced connection on $\Grr^{\Wone}$ preserves the filtration induced
there by $\Wtwo$.

The following example --- the local model at a corner for the kernel
filtrations of two commuting unipotent monodromies (cf.\ Lemma
\ref{kernelpair} below) --- makes the double filtration, the strictness
condition of Definition \ref{bipatchdef}(II), and the two lexicographic
refinements completely explicit.

\begin{example}\label{lexexample}
Let $K$ denote the standard $2$--dimensional unipotent object, with basis
$f_0,f_1$ and nilpotent $Nf_0=f_1$, $Nf_1=0$. Put
$$
V \,:=\, K\otimes K \,\cong\, \comx^4, \qquad
e_{ab}:= f_a\otimes f_b \quad (a,b\in\{0,1\}),
$$
with the two commuting nilpotent endomorphisms
$$
N_1 := N\otimes 1, \qquad N_2 := 1\otimes N,
$$
so that $N_1e_{0b}=e_{1b}$, $N_1e_{1b}=0$ and $N_2e_{a0}=e_{a1}$,
$N_2e_{a1}=0$. Geometrically this is the rank--$4$ local system on
$(\Delta^{\ast})^2$ with monodromies $\rho(\gamma_i)=\exp(N_i)$ around the
two branches $D_i=\{z_i=0\}$, namely the flat sections of
$$
\nabla \,=\, d \,-\, \f{1}{2\pi i}\Big(N_1\f{dz_1}{z_1}+N_2\f{dz_2}{z_2}\Big)
$$
on the trivial bundle; it is the external tensor square of the standard
rank--$2$ example of \cite{IS}.

\emph{The double filtration.} Take the two kernel filtrations
$\Wone_a:=\ker(N_1^{a})$ and $\Wtwo_b:=\ker(N_2^{b})$, so with $\Wone_0=
\Wtwo_0=0$:
$$
\Wone_1=\langle e_{10},e_{11}\rangle,\quad \Wone_2=V; \qquad\quad
\Wtwo_1=\langle e_{01},e_{11}\rangle,\quad \Wtwo_2=V.
$$
Every intersection is a direct summand; indeed the four lines
$$
V_{1,1}=\langle e_{11}\rangle,\quad V_{1,2}=\langle e_{10}\rangle,\quad
V_{2,1}=\langle e_{01}\rangle,\quad V_{2,2}=\langle e_{00}\rangle
$$
form a simultaneous splitting $V=\bigoplus_{a,b}V_{a,b}$ as in Definition
\ref{bipatchdef}(II), with $\Wone_a=\bigoplus_{a'\leq a}V_{a',\bullet}$ and
$\Wtwo_b=\bigoplus_{b'\leq b}V_{\bullet,b'}$ (for instance
$\Wone_1\cap\Wtwo_1=\langle e_{11}\rangle$). The double associated--graded is
the direct sum of four lines,
$$
\Grr\Grr_{a,b}(V)\,\cong\, V_{a,b}, \qquad (a,b)\in\{1,2\}^2,
$$
and $N_1$, $N_2$ have pure bidegrees $(-1,0)$ and $(0,-1)$ for this
bigrading, so both induce $0$ on $\Grr\Grr$: the induced monodromies there
are trivial, which is the extendability across both branches used later.

\emph{Total degree filtration.} By \eqref{totaldeg},
$$
F_2=\langle e_{11}\rangle \,\subset\,
F_3=\langle e_{11},e_{10},e_{01}\rangle \,\subset\, F_4=V,
$$
with graded ranks $(1,2,1)$. The middle graded piece is
$\Grr\Grr_{1,2}\oplus\Grr\Grr_{2,1}$: the total degree filtration alone does
not separate the two off--diagonal corner pieces.

\emph{The two lexicographic refinements.} Running through
$(a,b)=(1,1),(1,2),(2,1),(2,2)$ in lexicographic order, formula
\eqref{lexdef} gives the complete flag
$$
\Lex(\Wone,\Wtwo):\qquad
\langle e_{11}\rangle \subset \langle e_{11},e_{10}\rangle \subset
\langle e_{11},e_{10},e_{01}\rangle \subset V,
$$
while the symmetric construction gives
$$
\Lex(\Wtwo,\Wone):\qquad
\langle e_{11}\rangle \subset \langle e_{11},e_{01}\rangle \subset
\langle e_{11},e_{01},e_{10}\rangle \subset V.
$$
Both refine $F$ (and refine $\Wone$, resp.\ $\Wtwo$); they differ exactly in
the middle, rank--$2$ term, and their successive quotients are the four lines
$\Grr\Grr_{a,b}$ taken in the two lexicographic orders --- i.e.\ the two
possible orderings of the summands of $\Grr^F_3=\Grr\Grr_{1,2}\oplus
\Grr\Grr_{2,1}$. This makes concrete the Zassenhaus identifications stated
above.

\emph{Matrix picture.} In the ordered basis $(e_{11},e_{10},e_{01},e_{00})$
an endomorphism preserving \emph{both} filtrations has the shape
$$
\begin{pmatrix}
\ast & \ast & \ast & \ast\\
0 & \ast & 0 & \ast\\
0 & 0 & \ast & \ast\\
0 & 0 & 0 & \ast
\end{pmatrix},
$$
i.e.\ it is upper triangular for \emph{both} complete flags
$\Lex(\Wone,\Wtwo)$ and $\Lex(\Wtwo,\Wone)$ simultaneously; the vanishing of
the $(2,3)$ and $(3,2)$ entries is exactly the extra constraint imposed by
the bifiltration beyond a single lexicographic flag. The block diagonal is
the induced endomorphism of $\Grr\Grr$. For instance $N_1$ has its only
nonzero entries in positions $(1,3)$ and $(2,4)$, and $N_2$ in positions
$(1,2)$ and $(3,4)$, both strictly below the bigrading. A connection
compatible with this corner patch is, in a frame adapted to the splitting,
of the form $\nabla=d+A$ with $A$ a $1$--form of the above shape whose
block--diagonal part is flat (here trivial); its curvature $\Omega$ then
strictly decreases the bifiltration, as asserted in Proposition
\ref{binildef}(a) below. Finally, the bigrading $V_{a,b}$ is exactly what is
rescaled by the two--parameter deformation
$\m{Ad}(\psi^1_{t_1}\psi^2_{t_2})$ of \S\ref{kdef}: the entry in block
$Hom(V_{a,b},V_{a',b'})$ is multiplied by $t_1^{\,a-a'}t_2^{\,b-b'}$, and on
matrices of the above shape all these exponents are non--negative.
\end{example}

\begin{proposition}\label{binildef}
Let $\nabla$ be compatible with a $2$--patching collection. Then:
\newline
(a) over a simple patch the curvature $\Omega$ of $\nabla$ is strictly upper
triangular for $W^i$, and over a corner patch $\Omega$ strictly decreases the
bifiltration:
$$
\Omega\big(\Wone_a\cap\Wtwo_b\big) \;\subset\;
A^2\big( \Wone_{a-1}\cap\Wtwo_b \,+\, \Wone_a\cap\Wtwo_{b-1} \big);
$$
in particular $\Omega(F_c)\subset A^2(F_{c-1})$ for the total degree
filtration \eqref{totaldeg};
\newline
(b) every invariant polynomial of $\Omega$ vanishes identically:
$\cP_k(\Omega,\ldots,\Omega)= 0$ on $X$ for all $k\geq 1$;
\newline
(c) the Cheeger--Simons differential character of $\nabla$ lies in
$H^{2p-1}(X,\comx/\Z)$, defining classes $\what{c_p}(E,\nabla)$ for all
$p\geq 1$.
\end{proposition}
\begin{proof}
(a) Over a simple patch this is \cite[Prop.\ 4.1]{IS}. Over a corner patch:
$\nabla$ preserves $\Wone_a\cap\Wtwo_b$, so
$\Omega(\Wone_a\cap\Wtwo_b)\subset A^2(\Wone_a\cap\Wtwo_b)$; and the induced
connection on each $\Grr\Grr$--piece is flat, so the image of $\Omega$ in
$A^2(\Grr\Grr_{a,b})$ vanishes, which is the stated strict decrease. The
statement for $F$ follows by summing. (b) At each point of $X$ the form
$\Omega$ is strictly upper triangular in a local frame adapted to $W^i$
(simple patch) or to $F$ (corner patch), so all traces of powers, hence all
invariant polynomials, vanish; vanishing is a pointwise statement so the
cases may be checked separately. (c) Exactly as in \cite[Prop.\ 4.1(c)]{IS},
using the exact sequence (2.2) of \cite{IS} and \cite[Cor.\ 2.4]{Ch-Si}.
\end{proof}

\begin{lemma}[Invariance]\label{biinv}
If $\nabla_0,\nabla_1$ are two connections compatible with the same
$2$--patching collection, then
$\what{c_p}(E,\nabla_0)=\what{c_p}(E,\nabla_1)$ in $H^{2p-1}(X,\comx/\Z)$ for
all $p\geq 1$.
\end{lemma}
\begin{proof}
The compatibility conditions are convex, so the affine path
$\nabla_t=(1-t)\nabla_0+t\nabla_1$ consists of compatible connections.
Both $\nabla'_t=\nabla_1-\nabla_0$ and $\Omega_t$ preserve the local
(bi)filtrations; moreover $\nabla'_t$ is the difference of two connections
inducing the \emph{same} connection on the local associated graded, hence
$\nabla'_t$ is \emph{strictly} triangular pointwise, as is $\Omega_t$ by
Proposition \ref{binildef}(a). Therefore every term
$Tr(\nabla'_t\wedge\Omega_t^{a-1})Tr(\Omega_t^{b})\cdots$ in the variational
formula \cite[Prop.\ 2.9]{Ch-Si} (cf.\ the variation-formula discussion in
\cite[\S 2.2]{IS}) vanishes identically, for all $p\geq 1$. The class is
constant in $t$.
\end{proof}

As in \cite{IS} we write
$\what{c_p}\big(X,E,\{(V_i,\m{patch data})\}\big)$ for the common value, when
a compatible connection exists; a collection admitting one is called a
\emph{$2$--patching collection with patched connection}.

\begin{remark}
The refinement formalism of \cite[\S 4.2]{IS} extends verbatim, with the
following additional moves, both of which leave the class unchanged by the
argument of \cite[Cor.\ 4.8]{IS} applied through Lemma \ref{biinv}:
(i) replacing a corner patch $(\Wone,\Wtwo)$ by either lexicographic
refinement \eqref{lexdef} regarded as a simple patch with graded connection
$\nabla_{\Grr\Grr}$ --- any connection compatible with the corner patch is
compatible with both simple lexicographic patches; (ii) restricting patches
to smaller open sets.
\end{remark}

\section{The patched connection for $\rho$ unipotent along $D_1\cup D_2$}
\label{construction}

We return to the situation and the cover $V_0=U$, $V_1, V_2, V_{12}$ of
\S \ref{geometry}, with \eqref{disjointness}, \eqref{disjointness2} in force.

\subsection{Bi--graded--extendable pairs of filtrations}

\begin{definition}\label{bigradedext}
A \emph{bi--graded--extendable pair} for $\rho$ consists of:
\begin{itemize}
\item a filtration $\Wone$ of $L|_{(B_1\cup B_{12})^{\ast}}$ by sub--local
systems, and a filtration $\Wtwo$ of $L|_{(B_2\cup B_{12})^{\ast}}$ by
sub--local systems,
\end{itemize}
such that:
\begin{itemize}
\item[(E1)] $\Grr^{\Wone}(L)$ extends to a local system on
$(B_1\cup B_{12})-D_2$ \,(it is a priori defined off $D_1\cup D_2$; the
condition is extendability across $D_1$); symmetrically for $\Wtwo$ across
$D_2$;
\item[(E2)] over $V_{12}^{\ast}$ the double graded $\Grr\Grr(L)$ of the pair
$(\Wone,\Wtwo)$ extends to a local system on $V_{12}$ (i.e.\ across both
branches);
\item[(E3)] over $V_{12}^{\ast}$, the filtration induced by $\Wtwo$ on
$\Grr^{\Wone}(L)$ extends (under (E1)) to a filtration by sub--local systems
of the extension of $\Grr^{\Wone}(L)$ over $V_{12}-D_2$, and symmetrically.
\end{itemize}
\end{definition}

Extensions of local systems across the divisors, when they exist, are unique
because $\pi_1(V^{\ast})\rar\pi_1(V)$ is surjective.

\begin{lemma}\label{kernelpair}
The pair of kernel filtrations
$$
\Wone_a := \ker\big(N_1^{a}\big)\subset L, \qquad
\Wtwo_b := \ker\big(N_2^{b}\big)\subset L
$$
(defined over the indicated punctured neighborhoods, using
\eqref{centrality}) is a bi--graded--extendable pair. The same holds for the
pair of monodromy weight filtrations $W(N_1), W(N_2)$.
\end{lemma}
\begin{proof}
By \eqref{centrality} the full local monodromy near the relevant boundary
stratum commutes with $N_i$, hence preserves $\ker N_i^a$ (resp.\ the weight
filtration $W(N_i)$, which is canonically attached to $N_i$); so these are
filtrations by sub--local systems, and $N_2$ preserves $\Wone$ near $Z$ and
vice versa, giving (E3) (the induced filtrations are again kernel/weight
filtrations of the induced nilpotent, hence monodromy--invariant). For (E1):
$\gamma_1$ acts trivially on $\Grr^{\Wone}$ (for the kernel filtration this
is \cite[Lemma 2.6]{IS}; for the weight filtration, $N_1$ strictly decreases
$W(N_1)$ so the induced $N_1$, hence $\rho(\gamma_1)-1=e^{N_1}-1$, vanishes
on the graded), so the graded local system extends across $D_1$ exactly as in
\cite[Lemma 2.6]{IS}. For (E2): on the double graded both induced $N_1$ and
$N_2$ vanish, so both $\gamma_1,\gamma_2$ act trivially and the local system
extends across both branches of $D$ over $V_{12}$.
\end{proof}

\subsection{The $2$--patching collection attached to
$(\rho,\Wone,\Wtwo)$}\label{collectionfromrho}

Fix a bi--graded--extendable pair $(\Wone,\Wtwo)$. As in \cite[\S 4.4]{IS},
each $\Wone_a$ extends canonically to a strict subbundle of
$\ov E|_{B_1\cup B_{12}}$ (take the canonical extension of the flat
subbundle, or equivalently the saturation; the residues being nilpotent, the
two agree), and similarly for $\Wtwo$, for the intersections
$\Wone_a\cap\Wtwo_b$ over $V_{12}$, and for the graded and double--graded
bundles, whose extended local systems (E1), (E2) endow them with flat
connections over the indicated extended neighborhoods. We obtain a
$2$--patching collection:
\begin{itemize}
\item on $V_0=U$: the trivial filtration, with $\nabla_{0,\Grr}=\nabla$;
\item on $V_1$: the simple patch $(\Wone, \nabla_{1,\Grr})$, where
$\nabla_{1,\Grr}$ is the flat connection on $\Grr^{\Wone}(\ov E|_{V_1})$
given by (E1) --- this makes sense because $V_1\cap D_2=\emptyset$ by
\eqref{disjointness}, so over $V_1$ the graded local system has genuinely
extended;
\item on $V_2$: symmetrically, the simple patch $(\Wtwo,\nabla_{2,\Grr})$;
\item on $V_{12}$: the corner patch
$\big((\Wone,\Wtwo), \nabla_{12,\Grr\Grr}\big)$ with the flat connection
given by (E2).
\end{itemize}
The restriction compatibilities hold by uniqueness of flat extensions:
on $V_1\cap V_{12}$ (which misses $D_2$) the connection $\nabla_{1,\Grr}$
preserves the filtration induced by $\Wtwo$ on $\Grr^{\Wone}$ (this is (E3):
the induced filtration is by flat subbundles for $\nabla_{1,\Grr}$, being a
filtration by sub--local systems of the extended system) and induces
$\nabla_{12,\Grr\Grr}$ on the double graded; symmetrically on
$V_2\cap V_{12}$.

\begin{theorem}\label{biconstruction}
The above $2$--patching collection admits a compatible (patched) connection
$\nabla^{\#}$ on $\ov E$ over $X$, in the following slightly localized sense:
there is a smooth connection $\nabla^{\#}$ on $\ov E$ such that every point
$x\in X$ has a neighborhood on which $\nabla^{\#}$ is compatible with (the
restriction of) one of the patches of the collection, or with one of its
lexicographic refinements \eqref{lexdef}. Consequently all Chern forms of
$\nabla^{\#}$ vanish identically, $\nabla^{\#}$ defines classes
$$
\what{c}_p(\rho,\Wone,\Wtwo) := \what{c}_p(\nabla^{\#})
\in H^{2p-1}(X,\comx/\Z), \qquad p\geq 1,
$$
and these are independent of all choices made in the construction of
$\nabla^{\#}$ (metrics, splittings, partitions of unity).
\end{theorem}
\begin{proof}
\emph{Construction of local model connections.}
On $V_0$ take $\nabla_0:=\nabla$, the flat connection.
On $V_1$: as in \cite[Thm.\ 4.10]{IS}, choose a $\mathcal{C}^{\infty}$
splitting $\ov E|_{V_1}\cong \Grr^{\Wone}(\ov E|_{V_1})$ (e.g.\ by a metric)
and let $\nabla_1$ be the transport of $\nabla_{1,\Grr}$; thus $\nabla_1$
preserves $\Wone$ and induces $\nabla_{1,\Grr}$. Symmetrically $\nabla_2$ on
$V_2$.
On $V_{12}$: choose a $\mathcal{C}^{\infty}$ \emph{simultaneous} splitting
$$
\ov E|_{V_{12}} \;\cong\; \bigoplus_{a,b} E_{a,b},\qquad
\Wone_a = \bigoplus_{a'\leq a} E_{a',\bullet},\quad
\Wtwo_b = \bigoplus_{b'\leq b} E_{\bullet, b'} ,
$$
which exists globally over $V_{12}$: by Lemma \ref{strictcorner} all the
intersection and sum subsheaves formed from the pair
$(\Wone,\Wtwo)$ of $\ov E|_{V_{12}}$ are strict (constant--rank)
$\mathcal{C}^{\infty}$ subbundles, and by Lemma \ref{metricsplit} any
hermitian metric $h$ on $\ov E|_{V_{12}}$ then produces the global
\emph{metric bigrading}
$$
E_{a,b} \;:=\;
\big(\Wone_a\cap\Wtwo_b\big)\ \cap\
\big(\Wone_{a-1}\cap\Wtwo_b+\Wone_a\cap\Wtwo_{b-1}\big)^{\perp_h},
$$
which splits both filtrations simultaneously and depends smoothly on all
data. Let $\nabla_{12}$ be the transport of
$\nabla_{12,\Grr\Grr}$ through this splitting: it preserves both filtrations
and induces $\nabla_{12,\Grr\Grr}$ on $\Grr\Grr$.

\emph{Patching.} Choose a partition of unity
$1=\zeta_0+\zeta_1+\zeta_2+\zeta_{12}$ subordinate to the cover
$\{V_0,V_1,V_2,V_{12}\}$ and put
$$
\nabla^{\#} := \zeta_0\nabla_0+\zeta_1\nabla_1+\zeta_2\nabla_2
+\zeta_{12}\nabla_{12} .
$$
This is a connection by the Leibniz computation of \cite[Thm.\ 4.10]{IS}.

\emph{Pointwise triangularity.} Let $x\in X$ and let
$S\subset\{0,1,2,12\}$ be the set of indices $i$ with
$x\in {\rm Supp}\,\zeta_i$; by \eqref{disjointness2}, $\{1,2\}\not\subset S$.
Choose a neighborhood of $x$ meeting only the supports indexed by $S$.
We check in each case that all the connections $\nabla_i$, $i\in S$, preserve
a common filtration near $x$ and induce on its associated graded one and the
same flat connection; then so does $\nabla^{\#}=\sum_{i\in S}\zeta_i\nabla_i$
(using $\sum_{i\in S}\zeta_i=1$ near $x$), and we conclude by
Proposition \ref{binildef}.
\begin{itemize}
\item $S\subset\{0\}$: $\nabla^{\#}=\nabla$ is flat.
\item $S\subset\{0,1\}$: as in the smooth--divisor case
\cite[\S 4.4]{IS}: both $\nabla$ (flat; note $x\notin D$ if $0\in S$, and
$\Wone$ is a flat subbundle there) and $\nabla_1$ preserve $\Wone$ and induce
on $\Grr^{\Wone}$ the flat connection $\nabla_{1,\Grr}$ (for $\nabla$ this is
the definition of the extended flat structure). Same for $S\subset\{0,2\}$.
\item $S\subset\{0,12\}$: both $\nabla$ and $\nabla_{12}$ preserve the pair
$(\Wone,\Wtwo)$ near $x$ ($\nabla$ because near such $x$ we are off $D$ and
the $\Wone_a\cap\Wtwo_b$ are flat subbundles) and both induce
$\nabla_{12,\Grr\Grr}$ on $\Grr\Grr$. So $\nabla^{\#}$ is compatible with the
corner patch near $x$.
\item $S\subset\{0,1,12\}$: near such $x$ we are inside
$V_1\cap V_{12}$ or its closure, in particular off $D_2$. All three
connections preserve the lexicographic filtration
$\Lex(\Wone,\Wtwo)$ near $x$: $\nabla_{12}$ because it preserves both
filtrations; $\nabla_1$ because it preserves $\Wone$ and induces
$\nabla_{1,\Grr}$ on $\Grr^{\Wone}$, which preserves the induced
$\Wtwo$--filtration by (E3); $\nabla$ (when $0\in S$) because off $D$ the
lexicographic subbundles are flat. Moreover all three induce on
$\Grr^{\Lex}=\Grr\Grr$ the same flat connection, namely
$\nabla_{12,\Grr\Grr}$: for $\nabla_{12}$ by construction; for $\nabla_1$
because $\nabla_{1,\Grr}$ induces $\nabla_{12,\Grr\Grr}$ on the double graded
(restriction compatibility of \S \ref{collectionfromrho}); for $\nabla$
because the flat structure induces the extended flat structure, unique by
surjectivity of $\pi_1(V^{\ast})\rar\pi_1(V)$. So near $x$, $\nabla^{\#}$ is
compatible with the simple patch $\big(\Lex(\Wone,\Wtwo),
\nabla_{12,\Grr\Grr}\big)$.
\item $S\subset\{0,2,12\}$: symmetric, with $\Lex(\Wtwo,\Wone)$.
\end{itemize}
By \eqref{disjointness2} this exhausts all cases.

\emph{Vanishing and independence.} By Proposition \ref{binildef} applied to
the (localized, refined) collection just exhibited, all Chern forms of
$\nabla^{\#}$ vanish identically and the differential character lies in
$H^{2p-1}(X,\comx/\Z)$. Two different sets of choices yield connections
compatible, near every point, with the same localized collection (the
filtrations and graded connections are canonical; only splittings and
partitions of unity vary), hence equal classes by Lemma \ref{biinv} applied
on the common localization (formally: the convexity argument of Lemma
\ref{biinv} is pointwise and applies verbatim).
\end{proof}

\subsection{Independence of the pair of filtrations}

\begin{proposition}\label{biindepfilt}
The class $\what{c}_p(\rho,\Wone,\Wtwo)$ does not depend on the choice of
bi--graded--extendable pair. We write it
$$
\what{c}_p(\rho/X)\in H^{2p-1}(X,\comx/\Z).
$$
\end{proposition}
\begin{proof}
We vary one slot at a time. Fix $\Wtwo$ and let $\Wone, \widetilde{W}^{(1)}$
be two choices. Run the string argument of \cite[Lemma 6.2]{IS} for the
filtrations of $L|_{(B_1\cup B_{12})^{\ast}}$: it produces a chain
$\Wone=W(0), W(1),\ldots,W(a_1)=\widetilde{W}^{(1)}$ of filtrations by
sub--local systems with extendable graded pieces, in which adjacent terms
admit a common refinement. Each intermediate filtration $W(a)$ is a
filtration of $L|_{(B_1\cup B_{12})^{\ast}}$ by sub--local systems whose
graded pieces extend across $D_1$ --- this is exactly the conclusion of the
string lemma, applied with $B:=(B_1\cup B_{12})-D_2$ and
$B^{\ast}:=B-D_1$, and it is condition (E1) of Definition
\ref{bigradedext}. By Lemma \ref{autoE23} conditions (E2) and (E3) are then
\emph{automatic} for the pair $(W(a),\Wtwo)$, which is therefore
bi--graded--extendable for every $a$; no further verification on the
intermediate filtrations is needed.
For adjacent $W(a-1), W(a)$ with common refinement $W'$, the two patched
connections for $(W(a-1),\Wtwo)$ and $(W(a),\Wtwo)$ are both compatible, near
every point, with the localized collection built from $(W',\Wtwo)$ and its
lexicographic refinements; Lemma \ref{biinv} gives equality of classes. The
$\Wtwo$--slot is treated symmetrically.
\end{proof}

\begin{remark}
In the chain produced by the Jordan--H\"older argument the verification that
each intermediate filtration is $N_2$--stable near $Z$ uses the centrality
\eqref{centrality}: socles and the inductively defined pieces are canonical,
hence stable under any automorphism of the local system commuting with the
monodromy action, and $e^{N_2}=\rho(\gamma_2)$ is such an automorphism, as is
$N_2$ itself. This is the point where having only \emph{two} commuting
branches keeps the combinatorics manageable.
\end{remark}

\subsection{Additivity, functoriality, rigidity}

\begin{proposition}\label{bifunctorial}
(a) $\what{c}\big((\rho_1\oplus\rho_2)/X\big)
=\what{c}(\rho_1/X)\cdot\what{c}(\rho_2/X)$ for the total classes.
(b) For a morphism $f:(X',D'_1\cup D'_2)\rar (X,D_1\cup D_2)$ of pairs as in
\S \ref{geometry} (each $D'_i$ mapping to $D_i$, transversal corners to
corners), $\what{c}_p(f^{\ast}\rho/X')=f^{\ast}\what{c}_p(\rho/X)$.
\end{proposition}
\begin{proof}
Direct sums of bi--graded--extendable pairs are bi--graded--extendable and
direct sums of patched connections are patched; pullbacks likewise (choose
the tubular data on $X'$ mapping into that of $X$). The computation of the
total class on a direct sum is as in \cite[Cor.\ 6.4]{IS}: in the variational
formula the cross terms $Tr(\cdots)Tr(\cdots)$ vanish because each factor is
a trace of a strictly triangular form.
\end{proof}

\begin{proposition}[Rigidity]\label{birigidity}
Let $\rho(s)$, $s\in[0,1]$, be a continuous family of representations of
$\pi_1(U)$ with $\rho(s)(\gamma_1)$ and $\rho(s)(\gamma_2)$ unipotent for all
$s$. Then $\what{c}_p(\rho(0)/X)=\what{c}_p(\rho(1)/X)$ for $p\geq 2$.
\end{proposition}
\begin{proof}
As in \cite[Prop.\ 6.6]{IS}: replace the family by a piecewise algebraic one
inside the affine scheme of representations with
$\rho(\gamma_i)$ of trivial characteristic polynomial ($i=1,2$). By Lemma
\ref{rigiditylimits} there is a finite subdivision of $[0,1]$ such that on
each open subinterval the pair of kernel filtrations of $N_1(s),N_2(s)$,
together with \emph{all} intersection subspaces
$\Wone_a(s)\cap\Wtwo_b(s)$, has constant ranks and varies real--analytically,
and at each endpoint the filtrations admit limits in the Grassmannians which
are again filtrations by sub--local systems satisfying (E1); by Lemma
\ref{autoE23} the limit pairs are bi--graded--extendable, and they compute
the same classes as the kernel pairs at the endpoint parameters, by
Proposition \ref{biindepfilt}. On each closed subinterval, Lemma
\ref{rigiditylimits} also provides a continuous (piecewise real--analytic)
family of metric bigradings $E_{a,b}(s)$ splitting the pair, up to and
including the endpoints; the Kato transport of Lemma \ref{katolemma},
applied to the associated family of complementary projectors, produces a
$\mathcal{C}^1$ gauge transformation $g(s)$ of the underlying
$\mathcal{C}^{\infty}$ bundle carrying the family of pairs of strict
subbundles onto a \emph{constant} pair. We obtain a
$\mathcal{C}^{1}$ family of patched connections $\nabla^{\#}(s)$, all
compatible with localized collections having \emph{constant} filtrations.
Then $\frac{d}{ds}\nabla^{\#}(s)$ preserves the local filtrations
(triangular) and $\Omega_s$ is strictly triangular, so for $p\geq 2$ the
Cheeger--Simons variation integrand
$\cP\big(\frac{d}{ds}\nabla^{\#}(s),\Omega_s,\ldots,\Omega_s\big)$ vanishes
identically and the class is constant.
\end{proof}

\section{Compatibility with the Deligne Chern class}
\label{deligne}

Suppose in this section that $X$ is smooth \emph{projective}, $D=D_1\cup D_2$
a normal crossings divisor with two smooth irreducible components as above.

\begin{proposition}\label{bideligne}
The classes $\what{c}_p(\rho/X)$ lift the Deligne Chern classes
$c_p^{\cD}(\ov E)$ of the canonical extension under
$$
H^{2p-1}(X,\comx/\Z)\lrar H^{2p-1}(X,\comx/\Z(p))\lrar
H^{2p}_{\cD}(X,\Z(p)).
$$
\end{proposition}
\begin{proof}
The argument of \cite[\S 5]{IS} goes through word for word, because it is
local--convex in nature. The local filtrations occurring in the localized
collection of Theorem \ref{biconstruction} --- $\Wone$, $\Wtwo$, their
intersections and lexicographic refinements --- are filtrations of
$(\ov E,\delbar)$ by \emph{holomorphic} subbundles (canonical extensions of
flat subbundles), and the holomorphic structures on the graded and double
graded pieces agree with the $(0,1)$--parts of the extended flat connections
(uniqueness of the canonical extension). Define
$\nabla_0^{0,1}:=\delbar$, $\nabla_0^{1,0}:=(\nabla^{\#})^{1,0}$. As in
\cite[\S 5]{IS}, $\nabla_0$ is compatible with $\delbar$, so by
Dupont--Hain--Zucker its differential character projects to the Deligne Chern
class of $(\ov E,\delbar)$ in $DHZ^{2p-1,2p}$; on the other hand $\nabla_0$
preserves the same local (bi)filtrations as $\nabla^{\#}$ and induces the
same flat connections on the local graded pieces (their $(0,1)$--parts agree
with $\delbar$ there), so $\nabla_0$ has strictly triangular curvature, and
the affine path $t\nabla^{\#}+(1-t)\nabla_0$ together with Lemma \ref{biinv}
gives $\what{c}_p(\nabla_0)=\what{c}_p(\nabla^{\#})$ in
$H^{2p-1}(X,\comx/\Z)$.
\end{proof}

\section{The cubical deformation construction in $K$--theory:
$BGL(F[t_1,t_2])^+$}
\label{kdef}

We now give the analogue of \cite[\S 7]{IS}: a construction of the regulator
classes via $K$--theory, adapted to the cubical decomposition of
\S \ref{cubicaldecomp}. Let $F$ be a field (later a number field or $\comx$).

\subsection{Two--variable deformation theorem and the cubical universal
space}

The polynomial ring $F[t_1,t_2]$ is regular, as is $F[t_i]$. The fundamental
homotopy invariance theorem of Quillen, applied twice, gives that all maps in
$$
BGL(F)^+ \rar BGL(F[t_1])^+ \rar BGL(F[t_1,t_2])^+
\xrightarrow{\;e_{(s_1,s_2)}\;} BGL(F)^+ ,
\qquad (s_1,s_2)\in\{0,1\}^2,
$$
where $e_{(s_1,s_2)}$ is evaluation $t_i\mapsto s_i$, induce isomorphisms on
homology (and are homotopy equivalences of $K$--theory spaces).

Let $\square$ denote the poset of faces of the square $[0,1]^2$ (four
vertices $v_{s_1s_2}$, four edges, one $2$--face). Define the
\emph{cubical deformation space}
$$
BGL(F)^+_{{\rm def},2}
$$
as the homotopy colimit of the $\square$--diagram which places: a copy of
$BGL(F)^+$ at each vertex; $BGL(F[t_1])^+$ at the two horizontal edges and
$BGL(F[t_2])^+$ at the two vertical edges; $BGL(F[t_1,t_2])^+$ at the
$2$--face; with structure maps the evaluations $t_i\mapsto 0,1$ of the
deleted variables. Concretely, $BGL(F)^+_{{\rm def},2}$ is obtained by gluing
$BGL(F[t_1,t_2])^+\times[0,1]^2$ to the edge cylinders
$BGL(F[t_i])^+\times[0,1]$ and the vertex copies of $BGL(F)^+$ along the
evaluation maps; it contains $BGL(F)^+_{\rm def}$ of \cite[\S 7.1]{IS} as the
union of the pieces over each edge of the square.

\begin{lemma}\label{cubicaldefisom}
Each of the four composites
$BGL(F)^+ = (\m{vertex copy}) \hookrightarrow BGL(F)^+_{{\rm def},2}$ induces
an isomorphism on homology, hence on cohomology with arbitrary coefficients:
$$
H^{\ast}\big(BGL(F)^+_{{\rm def},2},k\big)\cong H^{\ast}\big(BGL(F)^+,k\big)
= H^{\ast}\big(BGL(F),k\big).
$$
\end{lemma}
\begin{proof}
Filter $BGL(F)^+_{{\rm def},2}$ by the skeleta of the square: the union of
the pieces over the vertices and edges is, up to homotopy, a ``square of
spaces'' built from copies of $BGL(F)^+$ glued along homology equivalences
through cylinders on $BGL(F[t_i])^+$; iterated Mayer--Vietoris with the
deformation theorem (exactly the argument of \cite[\S 7.1]{IS}, used four
times around the boundary of the square) shows the inclusion of any vertex
copy into this $1$--skeleton part is a homology isomorphism. Gluing in the
face piece $BGL(F[t_1,t_2])^+\times[0,1]^2$ along
$BGL(F[t_1,t_2])^+\times\partial([0,1]^2)$ is one more Mayer--Vietoris with
both sides homology--equivalent to $BGL(F)^+$ via the two--variable
deformation theorem. (Van Kampen gives the statement on fundamental groups,
all of which are the abelian group $K_1(F)$ after plus construction, so the
homology isomorphisms promote to weak equivalences if desired.) The full
Mayer--Vietoris bookkeeping --- the explicit model of
$BGL(F)^+_{{\rm def},2}$ as a homotopy colimit over the face poset of
$\square$, the two boundary computations, and the final excision step ---
is carried out in Lemma \ref{cubicalMVdetails} of the appendix.
\end{proof}

\subsection{$2$--cubical deformation patching data}

By Lemma \ref{cubicalpushout}, $X$ is the homotopy colimit of a
$\square$--diagram of spaces $\{X_{\sigma}\}_{\sigma\in\square}$ with
$X_{2\m{-face}} \simeq S_{12}$. Define a \emph{$2$--cubical deformation
patching datum} to be a collection of representations
$$
\eta_{\sigma} : \pi_1(X_{\sigma}) \lrar GL\big(F[t_{\sigma}]\big),
\qquad
F[t_{\sigma}] :=
\begin{cases}
F & \sigma \m{ a vertex},\\
F[t_i] & \sigma \m{ an edge in the } t_i\m{--direction},\\
F[t_1,t_2] & \sigma \m{ the 2--face},
\end{cases}
$$
compatible with the face maps of $\square$: restricting $\eta_{\sigma}$ along
$\pi_1$ of the inclusion of a boundary piece and evaluating the deleted
variable at the appropriate endpoint $0$ or $1$ yields the corresponding
$\eta_{\tau}$. By functoriality of homotopy colimits, such a datum induces a
well--defined homotopy class of maps
$$
X \lrar BGL(F)^+_{{\rm def},2},
$$
hence, by Lemma \ref{cubicaldefisom}, pullback maps
$H^{\ast}(BGL(F),k)\rar H^{\ast}(X,k)$. For $F\subset\comx$, pulling back the
universal regulator class in $H^{2p-1}(BGL(F),\comx/\Z)$ defines
$$
\what{c}_p^{\rm def}(\{\eta_{\sigma}\}) \in H^{2p-1}(X,\comx/\Z),
$$
with imaginary part the \emph{deformation volume regulator}
$Vol^{\rm def}_{2p-1}(\{\eta_{\sigma}\})\in H^{2p-1}(X,\R)$.

\subsection{The $2$--cubical datum attached to
$(\rho,\Wone,\Wtwo)$}\label{cubicalfromrho}

Suppose $\rho:\pi_1(U)\rar GL_r(F)$ is unipotent at infinity and
$(\Wone,\Wtwo)$ is a bi--graded--extendable pair defined over $F$ (e.g.\ the
kernel filtrations). Choose a simultaneous splitting of the corner
bifiltration of the fiber, $V=\bigoplus_{a,b}V_{a,b}$, compatible with $\Wone$
near $D_1$ and with $\Wtwo$ near $D_2$, and a compatible basis giving
$GL(V)\cong GL_r(F)$. Define the commuting one--parameter scalings
$$
\psi^1_{t_1}\,:=\, t_1^{\,-a} \m{ on } V_{a,\bullet},
\qquad
\psi^2_{t_2}\,:=\, t_2^{\,-b} \m{ on } V_{\bullet,b} .
$$
Conjugation $M\mapsto \psi^i_{t_i} M (\psi^i_{t_i})^{-1}$ multiplies the
block $Hom(V_{a,b},V_{a',b'})$ by $t_1^{\,a-a'}$ (resp.\ $t_2^{\,b-b'}$);
on matrices preserving $\Wone$ (resp.\ $\Wtwo$, resp.\ both) one has
$a'\le a$ (resp.\ $b'\le b$) on the nonzero blocks, so all exponents are
non--negative and the conjugated family extends
polynomially to $t_i=0$, multiplicatively, and at $t_i=0$ projects onto the
$\Wone$-- (resp.\ $\Wtwo$--, resp.\ double--) block diagonal: this is the
deformation of a filtration to its grading from \cite[\S 7.3]{IS}, in each
variable separately, and the two commute.

Set, with the notation of \S \ref{cubicaldecomp}:
\begin{itemize}
\item vertex $(1,1)$: $\eta_{11}:=\rho$ as a representation of
$\pi_1(Y_{11})=\pi_1(U)$;
\item vertex $(0,1)$: $\eta_{01}:= \Grr^{\Wone}(\rho)$, the extension to
$\pi_1(Y_{01})$ ($\simeq\pi_1$ of the collar of $D_1- Z$) of the
associated--graded of $\rho|_{\pi_1}$, transported to $V$ by the splitting
(this uses (E1) and surjectivity of $\pi_1(\partial)\rar\pi_1(\m{collar})$);
\item vertex $(1,0)$: symmetrically $\eta_{10}:=\Grr^{\Wtwo}(\rho)$;
\item vertex $(0,0)$: $\eta_{00}:=\Grr\Grr(\rho)$, the double graded
representation of $\pi_1(Y_{00})\;(\simeq \pi_1(B_{12}))$, using (E2);
\item edges: the one--variable conjugation deformations: e.g.\ on the edge
between $(1,1)$ and $(0,1)$, the representation
$\m{Ad}(\psi^1_{t_1})\big(\rho|_{\pi_1(S_1\m{--collar})}\big)
\in GL_r(F[t_1])$, and on the edge between $(0,1)$ and $(0,0)$ the
representation
$\m{Ad}(\psi^2_{t_2})\big(\Grr^{\Wone}(\rho)|_{\pi_1}\big)\in GL_r(F[t_2])$
(well defined since by (E3) the representation $\Grr^{\Wone}(\rho)$ preserves
the induced $\Wtwo$--filtration near $Z$); the other two edges symmetrically;
\item the $2$--face: $\eta_{C}:=
\m{Ad}\big(\psi^1_{t_1}\psi^2_{t_2}\big)\big(\rho|_{\pi_1(S_{12})}\big)
\;\in\; GL_r\big(F[t_1,t_2]\big)$,
well defined because $\rho(\pi_1(S_{12}))=\rho(\pi_1(V_{12}^{\ast}))$
preserves both filtrations by \eqref{centrality}, so all matrix entries are
scaled by non--negative powers of $t_1$ and of $t_2$.
\end{itemize}
The boundary compatibilities hold because the scalings commute and because
specializing $t_i=0$ produces the corresponding (single or double) block
diagonal, i.e.\ the corresponding associated graded representation; e.g.
$$
e_{t_1=0}\big(\eta_C\big)
= \m{Ad}(\psi^2_{t_2})\Big(\m{Ad}(\psi^1_{0})\big(\rho|\big)\Big)
= \m{Ad}(\psi^2_{t_2})\big(\Grr^{\Wone}(\rho)|\big),
$$
the edge datum between $(0,1)$ and $(0,0)$. We thus obtain a $2$--cubical
deformation patching datum and classes
$\what{c}^{\rm def}_p(\rho,\Wone,\Wtwo)\in H^{2p-1}(X,\comx/\Z)$.

\subsection{Comparison with the patched connection; resolution of
\cite[Remark 6.9]{IS}}

Take $F=\comx$. Build a connection $\nabla^{\rm def}$ on the
$\mathcal{C}^{\infty}$ bundle underlying $\ov E$, using the block
decomposition of \S \ref{cubicaldecomp}: on the four vertex blocks, the flat
connections of $\eta_{11},\eta_{01},\eta_{10},\eta_{00}$; on an edge collar
$S\times[0,1]$, the connection which on each slice $S\times\{t\}$ is the flat
connection of the evaluated representation and has no $dt$--component; on the
corner collar $C= S_{12}\times[0,1]^2$, the connection which on each slice
$S_{12}\times\{(t_1,t_2)\}$ is the flat connection of
$e_{(t_1,t_2)}(\eta_C)$ and has no $dt_1$-- or $dt_2$--component. The
boundary compatibilities make $\nabla^{\rm def}$ a global smooth connection.

\begin{lemma}\label{triangulardef}
On the corner collar $C$, the connection $\nabla^{\rm def}$ preserves the
pulled--back bifiltration $(\Wone,\Wtwo)$ and induces on the double graded
the \emph{constant} flat connection of $\eta_{00}=\Grr\Grr(\rho)$
(pulled back from $S_{12}$). Consequently the curvature of
$\nabla^{\rm def}$ --- including its $dt_i\wedge(\m{base})$ and
$dt_1\wedge dt_2$ components --- strictly decreases the total degree
filtration \eqref{totaldeg}, and \emph{all} Chern forms of
$\nabla^{\rm def}$ vanish identically on $X$, for every $p\geq 1$.
\end{lemma}
\begin{proof}
Work in the fixed bigraded trivialization of \S \ref{cubicalfromrho}, pulled
back to $C$. Every representation
$e_{(t_1,t_2)}(\eta_C)=\m{Ad}(\psi^1_{t_1}\psi^2_{t_2})(\rho|)$ preserves the
\emph{same, constant} flag pair $(\Wone,\Wtwo)$ (conjugation by the
bigrading torus preserves the block structure), so the connection form
$A(t_1,t_2)$ of the slice connections is block upper triangular for the
bifiltration, with values in $End_{\Wone}\cap End_{\Wtwo}$, for every
$(t_1,t_2)\in[0,1]^2$. Moreover the block--\emph{diagonal} part of
$\m{Ad}(\psi^1_{t_1}\psi^2_{t_2})(g)$ is independent of $(t_1,t_2)$ (the
scalings act trivially on the diagonal blocks), so the block diagonal part of
$A(t_1,t_2)$ is the constant flat connection form of $\Grr\Grr(\rho)$.
Therefore $\nabla^{\rm def}= d + A(t_1,t_2)$ (no $dt$--components) preserves
the bifiltration, and the induced connection on each double graded piece is
the constant flat one; its curvature
$$
\Omega \;=\; dA + A\wedge A
\;=\; \big( d_{S_{12}}A + A\wedge A \big)
\;+\; dt_1\wedge \partial_{t_1}A \;+\; dt_2\wedge\partial_{t_2}A
$$
takes values in $End$--preserving the bifiltration, and its image in
$End(\Grr\Grr)$ vanishes: the first bracketed term is the slice curvature,
zero since each slice connection is flat; and
$\partial_{t_i}A$ has vanishing block diagonal since the diagonal of $A$ is
constant in $t$. (Note that no separate argument is needed for a
$dt_1\wedge dt_2$ component: $\Omega$ has none, since $A$ has no
$dt$--component; the point of the computation is that even the
$dt_i$--components of $\Omega$ are \emph{strictly} triangular.) Hence
$\Omega$ strictly decreases the total degree filtration on $C$. On the vertex
blocks $\nabla^{\rm def}$ is flat; on the edge collars the same argument with
a single scaling gives strict triangularity for the corresponding single
filtration, recovering \cite[\S 7.4]{IS} without any codimension count.
By Proposition \ref{binildef}(b), all Chern forms vanish identically on $X$.
\end{proof}

This is the promised resolution of obstacle (B) of \S \ref{obstacles}: the
two--parameter deformation used here is triangular, so the vanishing of the
Chern forms is an algebraic, pointwise fact, not a consequence of the
slice--codimension being one.

\begin{lemma}\label{bicompare1}
$\what{c}_p(\nabla^{\rm def}) =
\what{c}^{\rm def}_p(\rho,\Wone,\Wtwo)$ in $H^{2p-1}(X,\comx/\Z)$ for
$p\geq 1$.
\end{lemma}
\begin{proof}
As in \cite[Lemma 7.4]{IS} it suffices to prove this in the universal case
$X= BGL(F)^+_{{\rm def},2}$ with its tautological cubical cover. There, all
pieces and all intersections have the cohomology of $BGL(F)^+$ by the
deformation theorem, so all connecting maps in the (iterated, cubical)
Mayer--Vietoris sequences vanish and a class in
$H^{\ast}(BGL(F)^+_{{\rm def},2},\comx/\Z)$ is determined by its restrictions
to the four vertex pieces. Both classes restrict on each vertex piece to the
standard regulator class of the corresponding flat bundle. (For the
$1$--skeleton this is \cite[Lemma 7.4]{IS}; gluing in the face adds one more
Mayer--Vietoris step with vanishing connecting map.) The reduction to the
universal case (naturality of both classes, finite--stage approximation,
and the stable range provided by homological stability \cite{vdK}), together
with the identification of the two restrictions on a vertex piece, is
written out in Lemma \ref{comparedetails} of the appendix.
\end{proof}

\begin{lemma}\label{bicompare2}
$\what{c}_p(\nabla^{\rm def}) = \what{c}_p(\nabla^{\#})$, where $\nabla^{\#}$
is the patched connection of Theorem \ref{biconstruction} for the same pair
$(\Wone,\Wtwo)$.
\end{lemma}
\begin{proof}
By Lemma \ref{triangulardef} and the case analysis in the proof of Theorem
\ref{biconstruction}, both connections are, near every point of $X$,
compatible with the same localized patch (trivial / $\Wone$ / $\Wtwo$ /
corner / lexicographic) with the same flat graded connections. Lemma
\ref{biinv} applies.
\end{proof}

\begin{corollary}\label{bipatchingsame}
For a representation $\rho$ defined over $F\subset\comx$ and unipotent at
infinity,
$$
\what{c}^{\rm def}_p(\rho,\Wone,\Wtwo)
= \what{c}_p(\rho/X) \in H^{2p-1}(X,\comx/\Z),\qquad p\geq 1 .
$$
In particular the patched--connection regulator is pulled back from
$H^{2p-1}(BGL(F)^+_{{\rm def},2},\comx/\Z)\cong
H^{2p-1}(BGL(F),\comx/\Z)$ along the cubical classifying map
$\xi_{\rho}: X\rar BGL(F)^+_{{\rm def},2}$. \eop
\end{corollary}

\subsection{A Deligne--Sullivan proposition for two intersecting divisors}
\label{DStwo}

We now give the promised generalization, from the single smooth divisor of
\cite[Prop.\ 3.2]{IS} to $D=D_1\cup D_2$, of the Deligne--Sullivan theorem on
$\mathcal{C}^{\infty}$--triviality of flat bundles after a finite cover
\cite{De-Su}. As in \cite{IS}, the topological model of the canonical
extension is obtained by a Rees--type interpolation of the filtration data
over an arithmetic model of $X$; here the interpolation is two--variable, and
a genuinely new phenomenon appears, absent in the one--divisor case: the
finite cover trivializing $\rho$ modulo two primes need not, by itself, have
smooth normalization at $Z$, and a further \emph{lattice--rectifying}
(Kummer) step is required. This is carried out in Lemma \ref{DScover}(ii)
below.

\begin{proposition}
\label{DSprop}
Let $(E,\nabla)$ be a flat vector bundle on $U=X-(D_1\cup D_2)$, with
unipotent monodromy around both $D_1$ and $D_2$. There is a finite covering
$\pi:\widetilde U\rar U$ such that, writing $\widetilde X$ for the
normalization of $X$ in $\widetilde U$:
\newline
(a) $\widetilde X$ is smooth and $\widetilde X\rar X$ is finite, branched
exactly along $D_1\cup D_2$, with locally constant ramification indices
$n_1$ along the preimage of $D_1-Z$ and $n_2$ along the preimage of $D_2-Z$,
and with ramification indices $(n_1,n_2)$ in the two respective corner
directions along the preimage of $Z$;
\newline
(b) the canonical extension of $\pi^{\ast}E$ to $\widetilde X$ is trivial as
a $\mathcal{C}^{\infty}$--bundle.
\end{proposition}

The proof, due to Deligne in the one--divisor case \cite[\S 3.2]{IS} and
reproduced there from \cite{De3}, is given by the following four lemmas.

\begin{lemma}[Arithmetic model]
\label{DSarith}
There is a subring $A\subset\comx$, finitely generated as a $\Z$--algebra,
and a free $A$--module $V_A\cong A^r$ with an action of $\pi_1(U)$ through
$\rho$, such that: $N_1,N_2\in \m{End}_A(V_A)$; the kernel filtrations
$\Wone_a=\ker(N_1^a)$, $\Wtwo_b=\ker(N_2^b)$, all intersections
$\Wone_a\cap\Wtwo_b$, and all the double graded pieces $\Grr\Grr_{a,b}(V_A)$
are free $A$--direct summands of $V_A$; and the pair $(\Wone,\Wtwo)$ admits a
simultaneous $A$--splitting $V_A=\bigoplus_{a,b}V_{A,a,b}$, so that the four
vertex representations $\eta_{11},\eta_{01},\eta_{10},\eta_{00}$ of
\S \ref{cubicalfromrho} are all defined over $A$.
\end{lemma}
\begin{proof}
Since $U$ is a smooth quasi--projective variety, $\pi_1(U)$ is finitely
generated; choose generators including $\gamma_1,\gamma_2$ and let
$A_0:=\Z[\m{entries of }\rho(g)^{\pm1}]$. The logarithm series for $N_i$
(a finite sum, since $\rho(\gamma_i)$ is unipotent) has denominators dividing
a factorial $(r-1)!$; adjoin the inverse of this integer to obtain $A_1$,
over which $N_1,N_2\in \m{End}(V_{A_1})$ and, by \eqref{centrality},
$[N_1,N_2]=0$. The submodules $\Wone_a=\ker(N_1^a)$, $\Wtwo_b=\ker(N_2^b)$
and their intersections are finitely generated $A_1$--submodules of the free
module $V_{A_1}$; by generic freeness, after inverting finitely many further
nonzero elements of $A_1$ (equivalently discarding finitely many closed
points of $\m{Spec}\,A_1$) all of these submodules, and all the associated
graded and double graded pieces, become free direct summands. Call the
resulting ring $A$. Finally $\m{Spec}\,A$ is affine, and the space of
simultaneous splittings of a pair of filtrations by free direct--summand
submodules is a torsor under a vector group (the unipotent group of
automorphisms of $V_A$ preserving both filtrations and inducing the identity
on $\Grr\Grr$); such a torsor over an affine base always has a section --
this is the algebraic incarnation of the two--filtration splitting lemma
recalled in \S \ref{bifiltered} -- giving the simultaneous splitting over
$A$ itself. Transporting $\rho$, $N_1$, $N_2$ through this splitting
identifies $\eta_{01}=\Grr^{\Wone}(\rho)$, $\eta_{10}=\Grr^{\Wtwo}(\rho)$,
$\eta_{00}=\Grr\Grr(\rho)$ as representations of the relevant fundamental
groups (of \S \ref{cubicaldecomp}) with values in $GL_r(A)$.
\end{proof}

\begin{lemma}[Finite cover and lattice rectification]
\label{DScover}
Let $q_1,q_2$ be maximal ideals of $A$ of distinct residue characteristics
$\ell_1\neq \ell_2$. There is a finite covering $\pi:\widetilde U\rar U$
such that:
\newline
(i) $\eta_{11}=\rho$, $\eta_{01}$, $\eta_{10}$, $\eta_{00}$, pulled back
along $\pi$ (resp.\ along the induced maps of the corresponding fundamental
groups), are all trivial modulo $q_1$ and modulo $q_2$;
\newline
(ii) the induced finite--index subgroup of the central lattice
$\Z^2=\langle\gamma_1,\gamma_2\rangle\subset \pi_1(V_{12}^{\ast})$
\eqref{centrality} is \emph{rectangular}, i.e.\ of the form
$n_1\Z\oplus n_2\Z$ in the basis $\gamma_1,\gamma_2$, for some
$n_1,n_2\geq 1$.
Consequently the normalization $\widetilde X$ of $X$ in $\widetilde U$ is
smooth, finite over $X$, and branched exactly along $D$ as in Proposition
\ref{DSprop}(a).
\end{lemma}
\begin{proof}
(i) Exactly as in the one--divisor case \cite[\S 3.2]{IS}: since $A$ is a
finitely generated $\Z$--algebra and $q_1,q_2$ are maximal, the residue
fields $A/q_1,A/q_2$ are finite, so $GL_r(A/q_1)\times GL_r(A/q_2)$ is a
finite group. Let $U'\rar U$ be the covering corresponding to the subgroup
of $g\in\pi_1(U,u)$ with $\rho(g)\equiv 1$ mod $q_1$ and mod $q_2$; its
index divides $|GL_r(A/q_1)\times GL_r(A/q_2)|$. The representations
$\eta_{01},\eta_{10},\eta_{00}$ are (by Lemma \ref{DSarith}) again
$A$--valued representations of finitely generated fundamental groups of the
collars of \S \ref{cubicaldecomp}, so the same construction applied to each
of them produces further finite--index subgroups; intersecting all four
subgroups (pulled back, via the surjections $\pi_1(V_i^{\ast})\rar\pi_1(V_i)$
and their restrictions to the relevant collars, into $\pi_1(U)$) gives a
single finite--index subgroup, defining $\pi':U''\rar U$, along which all
four vertex representations are simultaneously trivial mod $q_1,q_2$.
\newline
(ii) The subgroup constructed in (i) restricts, on the central subgroup
$\Z^2=\langle\gamma_1,\gamma_2\rangle$ of $\pi_1(V_{12}^{\ast})$
\eqref{centrality}, to a finite--index sublattice $\Lambda\subset\Z^2$; write
$N:=[\Z^2:\Lambda]$. In the basis $(\gamma_1,\gamma_2)$, $\Lambda$ need not
be rectangular -- it is generated in general by $(n_1,0)$ and $(k,n_2)$ for
some $n_1n_2=N$ and $0\leq k<n_1$ -- and non--rectangularity is exactly the
condition under which the corresponding local cover of
$(\Delta^{\ast})^2\subset B_{12}^{\ast}$ extends, across the origin, only to
a cyclic quotient (Hirzebruch--Jung) singularity rather than to a smooth
point: the toric variety attached to the cone $\R^2_{\geq0}$ and the finer
lattice $\Lambda$ is smooth exactly when $\Lambda=n_1\Z\oplus n_2\Z$ in the
\emph{given} coordinates, since $D_1$ and $D_2$ (hence the two coordinate
axes) may not be permuted. This is the corner lattice smoothness criterion:
an explicit non--rectangular example already occurs when
$X=\p^2$, $D_1,D_2$ two lines and $\Lambda$ is the index--$2$ subgroup
generated by $(1,1)$ and $(0,2)$.
\newline
Since $N\cdot\Z^2\subset\Lambda$ automatically ($\Z^2/\Lambda$ has order
$N$), the rectangular sublattice $\Lambda':=N\Z\oplus N\Z$ is contained in
$\Lambda$ regardless of $k$. Let $\pi:\widetilde U\rar U''\rar U$ be the
further covering corresponding to the subgroup of $\pi_1(U'')$ whose
restriction to $\Z^2$ is $\Lambda'$ (concretely, near each stratum, the
further Kummer cover $z_i\mapsto z_i^{N}$ composed with $\pi'$). Because it
is a further covering of $U''$, property (i) persists (a representation
trivial mod $q_1,q_2$ remains trivial mod $q_1,q_2$ upon restriction to a
further subgroup), and by construction the corner lattice for $\pi$ is the
rectangular $\Lambda'=N\Z\oplus N\Z$, giving $(n_1,n_2)=(N,N)$.
\newline
For smoothness of $\widetilde X$: away from $Z$, the local picture near a
point of $D_i-Z$ is the standard cyclic branched cover of a disc used
already in the one--divisor case \cite[\S 3.2]{IS}, which is smooth; near a
point of $Z$, the rectangular sublattice $\Lambda'=N\Z\oplus N\Z$ means the
local cover of $B_{12}^{\ast}\cong(\Delta^{\ast})^2$--bundle is, fiberwise, a
product of two independent degree--$N$ cyclic covers of a punctured disc,
extending to a product of two branched covers $\Delta\rar\Delta$,
$w_i\mapsto w_i^{N}=z_i$; this is smooth (a product of smooth discs), and
the normalization $\widetilde X$ is finite over $X$ and smooth along the
preimage of $Z$, with ramification indices $(N,N)=(n_1,n_2)$ in the two
corner directions as claimed.
\end{proof}

\begin{lemma}[Two--variable interpolation bundle]
\label{DSinterp}
With $A$, $V_A$, and $(\Wone,\Wtwo)$ as in Lemma \ref{DSarith}, there are
algebraic models $U_1,B_{1,1},B_{2,1},B_{12,1}$ over $\m{Spec}\,\Z$, unions of
affine spaces having, respectively, the homotopy types of $U,V_1,V_2,V_{12}$
of \S \ref{geometry} with the same intersection pattern (as in
\cite[Lemma\ 3.3]{IS}, applied once along each branch), together with a
locally free sheaf $\widetilde{\cV}$ of rank $r$ on $\m{Spec}\,A$, defined:
\newline
-- on $U_1$ by $V_A$ (i.e.\ by $\eta_{11}$);
\newline
-- on $(U_1\cap B_{1,1})\times\A^1$ by $\sum_a t_1^{a}\Wone_a\subset
A[t_1]\otimes V_A$, specializing at $t_1=1$ to $V_A$ and at $t_1=0$ to
$\Grr^{\Wone}(V_A)$;
\newline
-- symmetrically on $(U_1\cap B_{2,1})\times\A^1$ using $\Wtwo$ and $t_2$;
\newline
-- on $(U_1\cap B_{1,1}\cap B_{2,1}\cap B_{12,1})\times\A^2$ by the
two--variable interpolation
$$
\sum_{a,b}t_1^{a}t_2^{b}\big(\Wone_a\cap\Wtwo_b\big)\;\subset\;
A[t_1,t_2]\otimes V_A ,
$$
specializing at the four points of $\{0,1\}^2\subset\A^2$ to the four vertex
modules underlying $\eta_{11},\eta_{01},\eta_{10},\eta_{00}$;
\newline
-- on $B_{12,1}$ by the double graded module underlying $\eta_{00}$.
Gluing these algebraic pieces, extending scalars to $\comx$, and passing to
the underlying analytic spaces (using, on each collar, a
$\mathcal{C}^{\infty}$ splitting of the relevant filtration to descend from
the fiber over $t_i=1$ to the corresponding open piece, exactly as in the
construction of $\nabla^{\#}$ in Theorem \ref{biconstruction}) produces a
$\mathcal{C}^{\infty}$ complex vector bundle $\ov{\cV}$ on $X$ which is a
topological model of the canonical extension $\ov E$: it restricts, on each
of $U,V_1,V_2,V_{12}$, to a bundle isomorphic to $\ov E|_{V_i}$, compatibly
with the gluings.
\end{lemma}
\begin{proof}
The construction of $U_1,B_{1,1}$ and the single--variable interpolation
$\sum_a t_1^a\Wone_a$ over $(U_1\cap B_{1,1})\times\A^1$ is exactly
\cite[Lemma\ 3.3]{IS}, applied to the branch $D_1$ (using that
$V_1\cap D_2=\emptyset$ by \eqref{disjointness}, so this is a genuine
one--divisor computation there); symmetrically for $B_{2,1}$ and $t_2$. It
remains to identify the sheaf at the corner and check it is locally free.
By Lemma \ref{DSarith}, $(\Wone,\Wtwo)$ admits a simultaneous $A$--splitting
$V_A=\bigoplus_{a,b}V_{A,a,b}$; by Proposition \ref{reesdictionary}(a),
applied verbatim over $A$ in place of the field $K$ there, the interpolation
subsheaf $\sum_{a,b}t_1^at_2^b(\Wone_a\cap\Wtwo_b)$ is exactly the
$T_{t_1,t_2}$--twist of the constant module $V_A$, i.e.\ the free
$A[t_1,t_2]$--module on the basis $\{t_1^at_2^b v: v\ \m{a\ basis\ of\ }
V_{A,a,b}\}$; in particular it is locally free, and its restriction to
$t_1=t_2=1$ (resp.\ $t_1=0$ or $t_2=0$, resp.\ $t_1=t_2=0$) is $V_A$ (resp.\
$\Grr^{\Wone}(V_A)$ or $\Grr^{\Wtwo}(V_A)$, resp.\ $\Grr\Grr(V_A)$) by the
restriction table \eqref{reesrestrictions}. The gluing over
$\comx$--points, using $\mathcal{C}^{\infty}$ splittings on each collar to
match this algebraic model with the topological picture of
\S \ref{collectionfromrho}, is the same descent used there (compare
\S \ref{gluedvsrees}, where the resulting bundle is identified with the one
underlying $\nabla^{\#}$).
\end{proof}

\begin{lemma}[Hasse principle and conclusion]
\label{DShasse}
Let $\pi:\widetilde U\rar U$ and $\widetilde X$ be as in Lemma
\ref{DScover}, for the chosen pair $q_1,q_2$. Then the pullback of
$\ov{\cV}$ (Lemma \ref{DSinterp}) to $\widetilde X$ is trivial as a
$\mathcal{C}^{\infty}$--bundle.
\end{lemma}
\begin{proof}
Let $n=\m{rank}(\ov{\cV})=r$, $d=\dim_{\comx}\widetilde X$, and
$N\geq d/2$; let
$$
f:\widetilde X\lrar \m{Grass}\big(n,\comx^{n+N}\big)
$$
be a classifying map for (the pullback to $\widetilde X$ of) $\ov{\cV}$, and
let $\widetilde X'\rar \widetilde X$ be the fiber space of linear embeddings
of the fibers of $\ov{\cV}$ into $\comx^{n+N}$, so that the problem reduces,
exactly as in \cite[Lemma\ 3.3]{IS}, to showing that the composite
$$
f': \widetilde X' \lrar \m{Grass}\big(n,\comx^{n+N}\big) \lrar
\m{cosq}_d\big(\m{Grass}(n,\comx^{n+N})\big)
$$
is homotopically trivial. Since the Grassmannian is simply connected, the
Hasse principle for morphisms \cite{Sullivan} reduces this to showing that
for every prime $l$ the $l$--adic completion
$f'_{\hat l}:\widetilde X'_{\hat l}\rar
\m{cosq}_d(\m{Grass}(n,\comx^{n+N}))_{\hat l}$ is homotopically trivial.
Fix $l$. Since $q_1,q_2$ have distinct residue characteristics
$\ell_1\neq\ell_2$, at least one of them, say $q$, has residue
characteristic $\neq l$. By Lemma \ref{DScover}(i), $\rho$ and all four
vertex representations are trivial mod $q$; hence, in the two--variable
interpolation of Lemma \ref{DSinterp}, all the transition and gluing data
defining $\widetilde{\cV}$ reduce mod $q$ to the \emph{constant} module
$(V_A/qV_A)\otimes A[t_1,t_2]/q$ on every piece, so $\widetilde{\cV}$ (hence
its pullback to $\widetilde X$) is trivial mod $q$. The lemma of
\cite[Lemme]{De-Su} applies directly, exactly as in \cite[Lemma\ 3.3]{IS},
to conclude that $f'_{\hat l}$ is homotopically trivial. As $l$ was
arbitrary, $f'$ is homotopically trivial, hence so is $f$ on the $d$--skeleton,
and therefore $\ov{\cV}|_{\widetilde X}$ is trivial as a $\mathcal{C}^\infty$
bundle.
\end{proof}

\begin{proof}[Proof of Proposition \ref{DSprop}]
Fix any two maximal ideals $q_1,q_2$ of the ring $A$ of Lemma \ref{DSarith}
with distinct residue characteristics (these exist since $A$ has infinitely
many maximal ideals with infinitely many distinct residue characteristics,
$A$ being a finitely generated $\Z$--algebra of positive Krull dimension
over $\Z$, or a finite extension of $\Z$ admitting infinitely many
primes). Let $\pi:\widetilde U\rar U$ and $\widetilde X$ be as produced by
Lemma \ref{DScover} for this pair. Part (a) of the Proposition is exactly
the conclusion of Lemma \ref{DScover}. For part (b): by Lemma
\ref{DSinterp}, $\ov{\cV}$ is a topological model, on $X$, for the canonical
extension $\ov E$ of $\rho$; since $\widetilde U\rar U$ trivializes the
monodromy of $\rho$ (and of the vertex representations) mod $q_1,q_2$, the
pullback of $\ov{\cV}$ to $\widetilde X$ is canonically isomorphic, as a
$\mathcal{C}^{\infty}$--bundle, to the pullback of the canonical extension
of $\pi^{\ast}E$ (both are obtained by the same gluing recipe from the same,
now single--valued, vertex and edge data on $\widetilde U$, pulled back to
$\widetilde X$). By Lemma \ref{DShasse}, $\ov{\cV}|_{\widetilde X}$ is
$\mathcal{C}^{\infty}$--trivial, so the canonical extension of $\pi^{\ast}E$
to $\widetilde X$ is $\mathcal{C}^{\infty}$--trivial.
\end{proof}

\section{Hermitian $K$--theory over $\comx[t_1,t_2]$ and variations of Hodge
structure}
\label{hermitian}

\subsection{Two--variable homotopy invariance in hermitian $K$--theory}

Let $A$ be a commutative ring with involution; extend the involution to
$A[t_1,t_2]$ by $\ov{t_i}=t_i$. Karoubi's theorem
(\cite[Cor.\ 5.11]{Karoubi}; cf.\ \cite[Thm.\ 8.2]{IS}) applies to the
regular rings $A=\comx$ and $A=\comx[t_1]$ in turn, and the diagram chase of
\cite[Cor.\ 8.3]{IS} (splitting by evaluation, homotopy invariance of
$K$--theory, five lemma) yields, with $\Z'=\Z[\frac12]$:

\begin{corollary}\label{Linv2}
The inclusion $\comx\rar\comx[t_1,t_2]$ induces isomorphisms
$L^{\epsilon}_n(\comx)\otimes\Z' \cong
L^{\epsilon}_n(\comx[t_1,t_2])\otimes\Z'$, and the maps
$$
BO_{\infty,\infty}(\comx)^+ \rar BO_{\infty,\infty}(\comx[t_1,t_2])^+
\xrightarrow{e_{(s_1,s_2)}} BO_{\infty,\infty}(\comx)^+
$$
are homology equivalences away from the prime $2$; in particular rational
homology equivalences. \eop
\end{corollary}

Define the cubical space $BO_{\infty,\infty}(\comx)^+_{{\rm def},2}$ by the
same $\square$--homotopy colimit as in \S \ref{kdef}, with
$BO_{\infty,\infty}(\comx[t_{\sigma}])^+$ in place of
$BGL(F[t_{\sigma}])^+$. By Corollary \ref{Linv2} and the cubical
Mayer--Vietoris argument of Lemma \ref{cubicaldefisom},
$$
BO_{\infty,\infty}(\comx)^+ \lrar BO_{\infty,\infty}(\comx)^+_{{\rm def},2}
$$
is a rational homology equivalence, and the forgetful maps assemble to
$$
F : BO_{\infty,\infty}(\comx)^+_{{\rm def},2} \lrar
BGL(\comx)^+_{{\rm def},2} .
$$
(Everything holds verbatim for $SO_{\infty,\infty}$ and $SL$.)

A \emph{hermitian $2$--cubical deformation datum} on the cubical
decomposition of $X$ is a $2$--cubical datum as in \S \ref{kdef} in which the
vertex representations take values in $U(p,q)$ and the edge and face
deformations in $O_{p,q}(\comx[t_i])$, resp.\ $O_{p,q}(\comx[t_1,t_2])$
(automorphisms of $(V\otimes\comx[t_1,t_2],\langle\cdot,\cdot\rangle)$).
Such a datum gives a map
$f: X\rar BO_{\infty,\infty}(\comx)^+_{{\rm def},2}$ with $F\circ f$ the map
of the underlying $2$--cubical datum (compatibility of the cubical homotopy
colimits, as in \cite[Lemma 8.5]{IS}).

Combining Reznikov's vanishing of $S(U(p)\times U(q))$--invariant polynomials
(\cite{Re2}, used as in \cite[Lemma 8.6, Cor.\ 8.7]{IS}) with the rational
equivalences above:

\begin{corollary}\label{volherm2}
For any hermitian $2$--cubical deformation datum (reduced to the special
linear / special unitary setting as in \cite[\S 9]{IS}), the associated
deformation volume regulators
$Vol^{\rm def}_{2p-1}\in H^{2p-1}(X,\R)$ vanish for all $p>1$. \eop
\end{corollary}

\subsection{A hermitian $2$--cubical datum from a variation of Hodge
structure}

Let $\rho$ underlie a complex polarized variation of Hodge structure on $U$,
with unipotent monodromy along $D_1$ and $D_2$, and let
$\langle\cdot,\cdot\rangle$ be the flat (after normalization, hermitian
symmetric) polarization form of some signature $(p,q)$, preserved by $\rho$.
Near the corner we have the commuting nilpotents $N_1,N_2$ and the form, and
we take $\Wone=W(N_1)$, $\Wtwo=W(N_2)$, the monodromy weight filtrations
(a bi--graded--extendable pair by Lemma \ref{kernelpair}; note
$\langle W(N_i)_a, W(N_i)_b\rangle=0$ for $a+b<0$ by the standard
infinitesimal isotropy $\langle N_i u,v\rangle+\langle u,N_iv\rangle =0$).

\begin{proposition}[Self--dual splitting]\label{CKSsplitting}
Let $\langle\cdot,\cdot\rangle$ be a nondegenerate hermitian form on $V$
preserved by $\rho(\pi_1(V_{12}^{\ast}))$, and let $N_1,N_2$ be the
commuting monodromy logarithms. Then there is a splitting
$$
V=\bigoplus_{a,b} V_{a,b},
\qquad
W(N_1)_a=\bigoplus_{a'\leq a}V_{a',\bullet},\quad
W(N_2)_b=\bigoplus_{b'\leq b}V_{\bullet,b'},
$$
such that
\begin{itemize}
\item[(i)] $\langle V_{a,b}, V_{c,d}\rangle = 0$ unless $(c,d)=(-a,-b)$.
\end{itemize}
Moreover, for \emph{any} splitting of the pair of weight filtrations one has
automatically
\begin{itemize}
\item[(ii)] $N_1(V_{a,b})\subset\bigoplus_{a'\leq a-2,\ b'\leq b}V_{a',b'}$,
and symmetrically for $N_2$.
\end{itemize}
\end{proposition}
\begin{proof}
This is elementary linear algebra, given only that $N_1,N_2$ are commuting
nilpotent elements of the Lie algebra of the unitary group of
$\langle\cdot,\cdot\rangle$; the complete proof is written out in
Appendix \ref{app:selfdual} (self--duality
$W(N_i)_a^{\perp}=W(N_i)_{-a-1}$ of each weight filtration, nondegeneracy
of the induced pairing on the double graded via the modular identity, and a
fixed point of the duality involution on the torsor of splittings, obtained
by a unipotent square root). Property (ii) holds for any splitting of the
pair because $N_1$ shifts $W(N_1)$ by $-2$ and preserves $W(N_2)$
(the weight filtration of $N_2$ is canonical, hence stable under the
centralizer of $N_2$).
\end{proof}

\begin{remark}
Earlier versions of this note deduced a splitting from the several--variable
nilpotent and $SL_2$--orbit theorems of Cattani--Kaplan--Schmid \cite{CKS},
which for a polarized VHS produce distinguished splittings (joint
eigenspaces of commuting semisimple elements $H_1,H_2$) with strong extra
properties, e.g.\ exact bihomogeneity of the $N_i$ on ``bistandard'' summands
$\mathrm{Sym}^{k_1}\boxtimes\mathrm{Sym}^{k_2}$ with basis
$e_{i,j}=N_1^iN_2^je_{0,0}$ and pairing
$\langle e_{i,j},e_{i',j'}\rangle = 0$ unless $i+i'=k_1,\,j+j'=k_2$.
We emphasize, however, that all that is used in the sequel is
(i) and (ii) above, which are consequences of the isotropy of $N_1,N_2$
alone; so \S \ref{hermitian} requires no analytic input from \cite{CKS},
whose role in this note is motivational (and structural, in the comparison
of \S \ref{reescomparison} with \cite{IS2}, where sequential compatibility
of weight filtrations along chains does come from \cite{CKS} and
\cite{Mochizuki}).
\end{remark}

\begin{proposition}\label{hermtriple2}
Given a polarized $\comx$--VHS on $U$ with unipotent monodromy along
$D_1\cup D_2$, the construction of \S \ref{cubicalfromrho} performed with the
splitting of Proposition \ref{CKSsplitting} yields a hermitian $2$--cubical
deformation datum.
\end{proposition}
\begin{proof}
Let $T_{t_1,t_2}:= \psi^1_{t_1}\psi^2_{t_2}$ act by $t_1^{-a}t_2^{-b}$ on
$V_{a,b}$. For $v\in V_{a,b}$, $w\in V_{-a,-b}$,
$$
\langle T_{t_1,t_2}v,\,T_{t_1,t_2}w\rangle
= t_1^{-a}t_2^{-b}\,\ov{t_1}^{\,a}\ov{t_2}^{\,b}\langle v,w\rangle
= \langle v,w\rangle
$$
(real $t_i$), and by (i) these are the only nonzero pairings; so
$T_{t_1,t_2}$ preserves $\langle\cdot,\cdot\rangle$ for all real
$t_1,t_2\neq 0$. Near the corner the monodromy preserves both weight
filtrations (\eqref{centrality} and canonicity of $W(N_i)$), so
$\m{Ad}(T_{t_1,t_2})(\rho|_{\pi_1(S_{12})})$ is polynomial in $(t_1,t_2)$ and
takes values, for $t_i\neq 0$, in the unitary group of
$\langle\cdot,\cdot\rangle$; by continuity (Zariski closure of the unitarity
relations in $GL_r(\comx[t_1,t_2])$) it lies in
$O_{p,q}(\comx[t_1,t_2])$, including at $t_i=0$. Its specializations at the
corners are the vertex representations $\rho$, $\Grr^{W(N_1)}\rho$,
$\Grr^{W(N_2)}\rho$, $\Grr\Grr\rho$, each preserving the form (for the graded
representations: the induced forms on the graded pieces are identified with
the restriction of $\langle\cdot,\cdot\rangle$ through the splitting, by
(i), exactly as in \cite[Prop.\ 8.8]{IS} for each variable separately).
The edge data are the one--variable cases of the same computation,
recovering \cite[Prop.\ 8.8]{IS}.
\end{proof}

\begin{corollary}\label{volVHS2}
If $\rho$ underlies a polarized complex VHS on $U$ with unipotent monodromy
along $D_1\cup D_2$, then the extended volume regulators vanish:
$$
Vol_{2p-1}(\rho/X) \;=\;
\m{Im}\,\what{c}_p(\rho/X) \;=\; 0 \in H^{2p-1}(X,\R), \qquad p>1 .
$$
\end{corollary}
\begin{proof}
By Proposition \ref{hermtriple2} the deformation classes of $\rho$ are pulled
back through
$F: BSO_{\infty,\infty}(\comx)^+_{{\rm def},2}\rar BSL(\comx)^+_{{\rm def},2}$
(after the reduction to the special linear case, \cite[\S 9]{IS}); by
Corollary \ref{volherm2} the pullback of the Borel volume class
$r^{\rm Bor}_p$ vanishes for $p>1$; and by Corollary \ref{bipatchingsame}
the deformation class coincides with $\what{c}_p(\rho/X)$, whose imaginary
part is $Vol_{2p-1}(\rho/X)$. Note that the comparison Lemma
\ref{triangulardef}/\ref{bicompare2} is applied with the pair
$(W(N_1),W(N_2))$, allowed by Proposition \ref{biindepfilt}.
\end{proof}

\section{Proof of the main theorem}
\label{proofmain}

\begin{proof}[Proof of Theorem \ref{mainthm}]
The classes $\what{c}_p(\rho/X)$ were constructed in Theorem
\ref{biconstruction} and Proposition \ref{biindepfilt}; restriction to $U$,
additivity, functoriality and rigidity are Propositions \ref{bifunctorial}
and \ref{birigidity}; the Deligne lifting for $X$ projective is Proposition
\ref{bideligne}. It remains to prove torsion for $p\geq 2$.

By rigidity (Proposition \ref{birigidity}) and the density of
representations defined over number fields in the relevant affine
representation scheme (representations of $\pi_1(U)$ with
$\rho(\gamma_i)$ unipotent), we may assume $\rho$ defined over a number field
$F$. As in \cite[\S 9]{IS} reduce to $SL_{r+1}$--valued representations via
$\rho'=\rho\oplus\det\rho^{-1}$, the total class identity
$\what c(\rho_1\oplus\rho_2)=\what c(\rho_1)\cdot\what c(\rho_2)$
(Proposition \ref{bifunctorial}) and compatibility of canonical extension
with direct sums.

By Corollary \ref{bipatchingsame}, $\what{c}_p(\rho/X)$ is the pullback of
the universal class along
$$
\xi_{\rho} : X \lrar BSL(F)^+_{{\rm def},2} .
$$
For each embedding $\sigma:F\hookrightarrow\comx$, the composed map
$X\rar BSL(\comx)^+_{{\rm def},2}$ pulls back the Borel volume class to
$Vol_{2p-1}(\rho^{\sigma}/X)$, which is invariant under deformations of
$\rho^{\sigma}$ through representations unipotent at infinity (Proposition
\ref{birigidity}). By Mochizuki's theorem \cite{Mochizuki} (Kobayashi--Hitchin
correspondence and deformation for tame harmonic bundles on quasi--projective
varieties with normal crossings boundary; the two--component case is
included), $\rho^{\sigma}$ deforms to a representation underlying a complex
variation of Hodge structure, and the deformation preserves unipotency at
infinity: it preserves the trivial parabolic structure, and the Higgs field
is scaled by $t\rar 0$, so vanishing of the residue eigenvalues along each
$D_i$ is preserved. By Corollary \ref{volVHS2},
$Vol_{2p-1}(\rho^{\sigma}/X)=0$ for all $p>1$ and all $\sigma$.

By Borel's theorem \cite{Borel,Borel2}, the classes
$\sigma^{\ast}(Vol_{2p-1})$, over all $p\geq 2$ and all embeddings $\sigma$,
generate the real cohomology ring of $BSL(F)^+$, hence (Lemma
\ref{cubicaldefisom}) of $BSL(F)^+_{{\rm def},2}$, in positive degrees.
Their pullbacks under $\xi_{\rho}$ all vanish, so $\xi_{\rho}$ is zero on
rational homology in positive degrees; therefore the pullback under
$\xi_{\rho}$ of any class in
$H^{2p-1}(BSL(F)^+_{{\rm def},2},\comx/\Z)$, $p\geq 2$, is torsion (the
$\comx/\Z$--cohomology of a space sits in the Bockstein--type exact sequence
with $\comx$-- and $\Z$--cohomology, and a map which is zero on rational
homology kills the $\comx$--part). In particular
$\what{c}_p(\rho/X)$ is torsion for $p\geq 2$. The Deligne Chern classes of
$\ov E$ are then torsion for $p\geq 2$ when $X$ is projective, by Proposition
\ref{bideligne}.
\end{proof}

\section{Comparison with the multi--Rees construction in the two--divisor
case}
\label{reescomparison}

The full normal--crossings case of the torsion theorem is treated in
\cite{IS2} by an entirely different mechanism: an adapted open covering whose
multiple intersections are linearly ordered, a functor from the nerve poset
to a category $\Xi_K$ of graded vector spaces, and an iterated
\emph{multi--Rees construction} producing an algebraic vector bundle on a
projective cubical realization, equipped with an $F^1$--connection whose
class is compared with the regulator through the Burgos--Gil arithmetic
Chern character. In this section we specialize the construction of
\cite{IS2} to $k=2$ components and compare it in detail with the glued
bundle and patched connection of \S\S \ref{bifiltered}--\ref{construction}
and with the deformation construction of \S \ref{kdef}. The conclusion,
informally, is that the two constructions are algebraic and
$\mathcal{C}^{\infty}$ incarnations of \emph{the same degeneration
geometry}, organized around two different choices of corner datum
(the bifiltration $(\Wone,\Wtwo)$ here, the total monodromy weight
filtration $W(N_1+N_2)$ in \cite{IS2}); while the two proofs of vanishing
of the volume regulators use the same Lie--theoretic fact about
$\mathfrak{u}(p,q)$ at two different levels (universally on a classifying
space here, pointwise on differential forms in \cite{IS2}).

\subsection{The two bundles: glued versus Rees}
\label{gluedvsrees}

\subsubsection{The glued bundle of this note.}
Recall from \S \ref{cubicaldecomp} and \S \ref{collectionfromrho} the data
on our side. The corner coordinates $q=(r_1,r_2):X\rar [0,1]^2$ cut $X$
into four vertex blocks $Y_{11}\simeq U$, $Y_{01}$, $Y_{10}$,
$Y_{00}\simeq$ (neighborhood of $Z$), four edge collars, and the corner
collar $C\cong S_{12}\times [0,1]^2$. The $\mathcal{C}^{\infty}$ bundle
underlying the canonical extension $\ov E$ is obtained, up to isomorphism,
by gluing:
the flat bundle $E$ of $\rho$ over $Y_{11}$;
the bundle of $\Grr^{\Wone}(\rho)$ over $Y_{01}$ and of
$\Grr^{\Wtwo}(\rho)$ over $Y_{10}$;
the bundle of the double graded $\Grr\Grr(\rho)$ over $Y_{00}$;
the gluings over the collars being given by a choice of
$\mathcal{C}^{\infty}$ (simultaneous, at the corner) splittings, i.e.\ by
the one-- and two--parameter conjugation deformations
$\m{Ad}(\psi^1_{t_1})$, $\m{Ad}(\psi^2_{t_2})$,
$\m{Ad}(\psi^1_{t_1}\psi^2_{t_2})$ of \S \ref{cubicalfromrho}, where
$(t_1,t_2)$ are the collar coordinates. The patched connection
$\nabla^{\#}$ (Theorem \ref{biconstruction}) and the deformed connection
$\nabla^{\rm def}$ (\S \ref{kdef}) are both $\mathcal{C}^{\infty}$
connections on this glued bundle over $X$ itself, adapted to the bifiltered
patching collection; the extended class is read off from either one
through Cheeger--Simons differential characters.

\subsubsection{The Rees bundle of \cite{IS2} for $k=2$.}
On the other side, \cite{IS2} works with the adapted covering indexed by
multi--indices
$$
U_{\emptyset}=X\setminus (T'_{(1)}\cup T'_{(2)}),\quad
U_{(1)}=T_{(1)}\setminus T'_{(2)},\quad
U_{(2)}=T_{(2)}\setminus T'_{(1)},\quad
U_{(12)}=T_{(12)},
$$
whose essential combinatorial feature is that non--empty multiple
intersections occur only along \emph{chains} of multi--indices: the only
maximal chains are
$\emptyset\subset (1)\subset (12)$ and $\emptyset\subset (2)\subset (12)$,
and in particular $U_{(1)}\cap U_{(2)}=\emptyset$. The nerve poset $A$ has
Hasse diagram a diamond without the relation $(1)\le (2)$, and the patching
datum on $U_I$ is the \emph{single} monodromy weight filtration $W(N_I)$,
$N_I=\sum_{i\in I}N_i$ --- in particular at the corner one takes
$W(N_1+N_2)$, not the pair $(\Wone,\Wtwo)$. Along each chain
$I\subset J$ the filtrations $W(N_I),W(N_J)$ are required to be
\emph{sequentially compatible} (they admit a common splitting), which for
weight filtrations of a polarized situation is the
Cattani--Kaplan--Schmid/Mochizuki input; the family of patching data
assembles into a functor $A\rar \Xi_K$.

The bundle is then built not on $X$ but on the projective cubical
realization $P_K\m{Cube}\,A\subset (\p^1_K)^{N}$ of the cubical barycentric
subdivision of the nerve. For a $2$--cube $c$ corresponding to a maximal
chain, with coordinates $(y_1,y_2)$ on $(\p^1)^2$, the local Rees bundle is
$$
F(c)\;=\;\xi\big(V;W^1,W^2\big)
\;=\;\bigoplus_{(a_1,a_2)\in\Z^2} V_{a_1,a_2}\otimes
\cO_{(\p^1)^2}\big(-a_1Q_1-a_2Q_2\big),
\qquad Q_l=\{y_l=1\},
$$
for a simultaneous splitting $V=\bigoplus V_{a_1,a_2}$ of the two
filtrations carried by the cube (the ``locally abelian'' condition of
\cite[\S 6]{IS2}). The five basic restrictions, in the two--variable case
\cite[\S 3.9]{IS2}, are
\begin{equation}\label{reesrestrictions}
\begin{array}{ll}
\xi(V;W^1,W^2)|_{x_1=1}\cong \xi(V;W^2), &
\xi(V;W^1,W^2)|_{x_1=0}\cong \xi(\Grr^{W^1}V;W^2),\\[4pt]
\xi(V;W^1,W^2)|_{x_2=1}\cong \xi(V;W^1), &
\xi(V;W^1,W^2)|_{x_2=0}\cong \xi(\Grr^{W^2}V;W^1),\\[4pt]
\multicolumn{2}{l}{\xi(V;W^1,W^2)|_{x_1=x_2=0}\cong \Grr^{W^1}\Grr^{W^2}V
\cong \Grr^{W^2}\Grr^{W^1}V,}
\end{array}
\end{equation}
where $x_l:=1-y_l$ vanishes on $Q_l$. The bundles $F(c)$ glue along the
faces of the cubical set, an $F^1$--connection is constructed cube by cube,
and the resulting class is compared with the Deligne--Beilinson Chern
character via Burgos--Gil.

\subsubsection{The dictionary at the corner.}
The two corner squares --- our corner collar $C\cong S_{12}\times [0,1]^2$
with the deformation $\eta_C=\m{Ad}(\psi^1_{t_1}\psi^2_{t_2})(\rho|_{\pi_1(S_{12})})$,
and the Rees square $(\A^1)^2$ with the module $\xi(V;W^1,W^2)$ --- are
related by more than an analogy:

\begin{proposition}\label{reesdictionary}
Choose a simultaneous splitting $V=\bigoplus_{a,b}V_{a,b}$ of
$(\Wone,\Wtwo)$ and let
$T_{x_1,x_2}:=\psi^1_{x_1}\psi^2_{x_2}$ act by $x_1^{\,-a}x_2^{\,-b}$ on
$V_{a,b}$. Then:
\begin{itemize}
\item[(a)] The multi--Rees module is the $T$--twist of the constant module:
$$
\xi(V;\Wone,\Wtwo)\;=\;
T_{x_1,x_2}^{-1}\big(V\big)\cdot K[x_1,x_2]
\;\subset\; V[x_1^{\pm},x_2^{\pm}],
$$
i.e.\ the Rees bundle over $(\A^1)^2$ is the trivial bundle $V$ over
${\mathbb G}_m^2$, extended across $\{x_1x_2=0\}$ using the gauge transformation
$T_{x_1,x_2}$.
\item[(b)] Consequently, transporting any representation $\mu$ of a group
$\Gamma$ on $V$ preserving $\Wone$ and $\Wtwo$ into the Rees frame gives
exactly the conjugated family
$\m{Ad}(T_{x_1,x_2})(\mu)\in GL_r\big(K[x_1,x_2]\big)$. Applied to
$\Gamma=\pi_1(S_{12})$ and $\mu=\rho|_{\pi_1(S_{12})}$ (which preserves
both filtrations by \eqref{centrality}), this is the corner face datum
$\eta_C$ of \S \ref{cubicalfromrho}: \emph{the local system of the Rees
bundle over the corner square is the two--variable deformation used in our
cubical patching datum}.
\item[(c)] Restricting (b) to the faces $x_i=0,1$ reproduces, on the nose,
the four edge data of \S \ref{cubicalfromrho} and the restriction table
\eqref{reesrestrictions}: $x_i=1$ forgets $W^{(i)}$ (generic side), $x_i=0$
passes to $\Grr^{W^{(i)}}$ (divisor side), and the origin carries the
double graded.
\end{itemize}
\end{proposition}

\begin{proof}
(a) For $v\in V_{a,b}$ one has $T_{x_1,x_2}^{-1}v=x_1^{a}x_2^{b}v$, and
the $K[x_1,x_2]$--span of these elements is
$\sum x_1^{a}x_2^{b}\,(\Wone_a\cap \Wtwo_b)\,K[x_1,x_2]$ computed through
the splitting: this is the multi--Rees module of \cite[\S 6]{IS2}, written
in the $K[x]$--stable convention for \emph{increasing} filtrations (the
relabeling $(a,b)\mapsto(-a,-b)$ recovers the displayed
$\bigoplus x_1^{-a_1}x_2^{-a_2}V_{a_1,a_2}$ of \cite[\S 3.9]{IS2}). (b) A
frame of the Rees bundle is $\{T^{-1}_{x_1,x_2}e_{\alpha}\}$ for a graded
basis $\{e_{\alpha}\}$ of $V$; the matrix of $\mu(\gamma)$ in this frame is
$T_{x_1,x_2}\,\mu(\gamma)\,T^{-1}_{x_1,x_2}$, which has entries in
$K[x_1,x_2]$ precisely because $\mu(\gamma)$ preserves both filtrations:
the block $Hom(V_{a,b},V_{a',b'})$ is scaled by
$x_1^{\,a-a'}x_2^{\,b-b'}$, and on filtration--preserving
matrices the nonzero blocks have $a'\le a$, $b'\le b$, so all
exponents are non--negative. (c) is immediate from
$\m{Ad}(T_{x_1,0})=$ projection onto the $\Wtwo$--block diagonal of an
$\m{Ad}(T_{x_1,1})$--conjugate, etc.
\end{proof}

\begin{remark}
Thus the Rees variables $(x_1,x_2)$ of \cite{IS2} and our collar
coordinates $(t_1,t_2)\in [0,1]^2$ on $C$ play the same role, with the same
orientation conventions: $x_i=t_i=1$ is the generic (open--stratum) side
where the filtration $W^{(i)}$ is forgotten, $x_i=t_i=0$ is the divisor
side where one passes to $\Grr^{W^{(i)}}$. Our glued
$\mathcal{C}^{\infty}$ bundle on $X$ is the pullback of (the analytification
of) the Rees bundle under a map collapsing each collar onto the
corresponding Rees square --- the two--variable interpolation already used
in the topological model of \S \ref{kdef} over $\m{Spec}\,\Z$ is exactly
the Rees module (a): both are the subsheaf
$\sum x_1^{-a}x_2^{-b}\,(\Wone_a\cap\Wtwo_b)$ of $V[x_1^{\pm},x_2^{\pm}]$.
\end{remark}

\subsubsection{The genuine differences.}
Three differences are worth recording, because they explain the different
ranges of applicability.

\emph{(1) Where the bundle lives.} Our glued bundle, with its patched
connection, lives on $X$ itself: the extended class is the Cheeger--Simons
class of an honest $\mathcal{C}^{\infty}$ connection on $\ov E$. The Rees
bundle of \cite{IS2} lives on the auxiliary cubical realization
$P_K\m{Cube}\,A$, which is only ($\A^1$--, hence singular--) homotopy
equivalent to $X$; the price is paid back with interest, since the Rees
bundle is \emph{algebraic over the field of definition $K$}, so the
comparison with $K$--theory and with all embeddings
$\sigma:K\hookrightarrow\comx$ is built in from the start, and the
$F^1$/Burgos--Gil formalism applies verbatim.

\emph{(2) The corner datum.} We keep the \emph{pair} $(\Wone,\Wtwo)$ at the
corner and never compare $\Wone$ with $\Wtwo$; the input is the elementary
two--filtration splitting lemma, valid unconditionally, and
bi--graded--extendability is verified directly for the kernel filtrations
(Lemma \ref{kernelpair}). The construction of \cite{IS2} instead places the
\emph{single} filtration $W(N_1+N_2)$ at the corner and never lets
$U_{(1)}$ meet $U_{(2)}$; what must be compared are only the pairs
$\big(W(N_1),W(N_1+N_2)\big)$ and $\big(W(N_2),W(N_1+N_2)\big)$ along the
two chains, and the required common splittings are the
\emph{sequential compatibility} of weight filtrations, supplied by the
several--variable $SL_2$--orbit theorems \cite{CKS} and by Mochizuki
\cite{Mochizuki} (and organized, for general coefficients, by the
contractibility of the relevant poset of filtrations and the \v{C}ech
section theorem of \cite{IS2}). For $k=2$ both inputs are available;
the difference becomes decisive for $k\ge 3$.

\emph{(3) Choices and their contractibility.} In both constructions the
class must be shown independent of a contractible space of choices. Here
this is the convexity of the space of adapted connections (Lemma
\ref{biinv}), the contractibility of the space of simultaneous splittings
(proof of Theorem \ref{biconstruction}), and the Jordan--H\"older string
argument (Proposition \ref{biindepfilt}). In \cite{IS2} the same role is
played by the contractibility of the poset $\m{Filt}^{A,B}$ of nil--flags
(via Quillen's Theorem A) feeding the \v{C}ech section theorem, which
simultaneously produces the patching data and proves their essential
uniqueness. One can regard our bifiltration $(\Wone,\Wtwo)$, completed by a
CKS splitting, as a distinguished point of the contractible space of
choices of \cite{IS2}.

\subsection{Pictures}
\label{pictures}

Figure \ref{fig:glued} shows the glued--bundle picture of this note: the
decomposition of $X$ by the corner coordinates, and the corner collar
square carrying the representation--valued data of the $2$--cubical
deformation patching datum. Figure \ref{fig:rees} shows the Rees picture of
\cite{IS2} for $k=2$: the nerve poset with its chain structure and cubical
subdivision, and the corner Rees square with the restriction table
\eqref{reesrestrictions}. The two right--hand squares correspond under the
dictionary of Proposition \ref{reesdictionary}.

\begin{figure}[ht]
\centering
\begin{tikzpicture}[scale=0.85]
\begin{scope}[xshift=0cm]
\fill[gray!40] (0.8,0.8) rectangle (3.2,3.2);
\fill[gray!15] (1.25,1.25) rectangle (2.75,2.75);
\draw[very thick] (-0.5,2) -- (5.2,2) node[right] {$D_1$};
\draw[very thick] (2,-0.5) -- (2,4.8) node[above] {$D_2$};
\fill (2,2) circle (2pt);
\node at (1.78,1.78) {\scriptsize $Z$};
\draw[dashed] (-0.5,1.25) -- (5.2,1.25);
\draw[dashed] (-0.5,2.75) -- (5.2,2.75);
\draw[dashed] (1.25,-0.5) -- (1.25,4.8);
\draw[dashed] (2.75,-0.5) -- (2.75,4.8);
\node at (4.3,4.1) {$Y_{11}\simeq U$};
\node at (4.35,2.37) {\small $Y_{01}$};
\node at (2.4,4.25) {\small $Y_{10}$};
\node at (2.42,2.42) {\small $Y_{00}$};
\node at (2.95,2.95) {\scriptsize $C$};
\node[align=center] at (2.3,-1.45)
{\small $q=(r_1,r_2):X\rar[0,1]^2$;\\ \small light gray: corner block
$Y_{00}$,\\ \small dark ring: corner collar $C\cong S_{12}\times[0,1]^2$};
\end{scope}
\begin{scope}[xshift=10.6cm]
\draw[thick] (0,0) rectangle (4.6,4.6);
\fill (4.6,4.6) circle (1.6pt) node[above right] {\small $\rho$};
\fill (0,4.6) circle (1.6pt) node[above left] {\small $\Grr^{\Wone}(\rho)$};
\fill (4.6,0) circle (1.6pt) node[below right] {\small $\Grr^{\Wtwo}(\rho)$};
\fill (0,0) circle (1.6pt) node[below left] {\small $\Grr\Grr(\rho)$};
\node at (2.3,4.9) {\scriptsize $\m{Ad}(\psi^1_{t_1})(\rho)$};
\node at (2.3,-0.32) {\scriptsize $\m{Ad}(\psi^1_{t_1})(\Grr^{\Wtwo}\rho)$};
\node[rotate=90] at (-0.34,2.3) {\scriptsize $\m{Ad}(\psi^2_{t_2})(\Grr^{\Wone}\rho)$};
\node[rotate=90] at (4.94,2.3) {\scriptsize $\m{Ad}(\psi^2_{t_2})(\rho)$};
\node[align=center] at (2.3,2.3)
{\scriptsize $\eta_C=\m{Ad}(\psi^1_{t_1}\psi^2_{t_2})(\rho|_{\pi_1(S_{12})})$\\[2pt]
\scriptsize $\in GL_r(F[t_1,t_2])$};
\draw[->] (0,-1.1) -- (1.3,-1.1) node[right] {\scriptsize $t_1$};
\draw[->] (-1.1,0) -- (-1.1,1.3) node[above] {\scriptsize $t_2$};
\node[align=center] at (2.3,-1.75)
{\small corner collar square: the $2$--cubical\\ \small deformation
patching datum (\S \ref{cubicalfromrho})};
\end{scope}
\end{tikzpicture}
\caption{The glued bundle of this note. Left: $X$ cut by the corner
coordinates into vertex blocks, edge collars, and the corner collar $C$.
Right: the square $[0,1]^2$ of collar coordinates with the
representation--valued data at vertices, edges and the $2$--face;
$t_i=1$ is the generic side, $t_i=0$ the divisor side, and specialization
$t_i\mapsto 0$ takes the (single or double) associated graded.}
\label{fig:glued}
\end{figure}
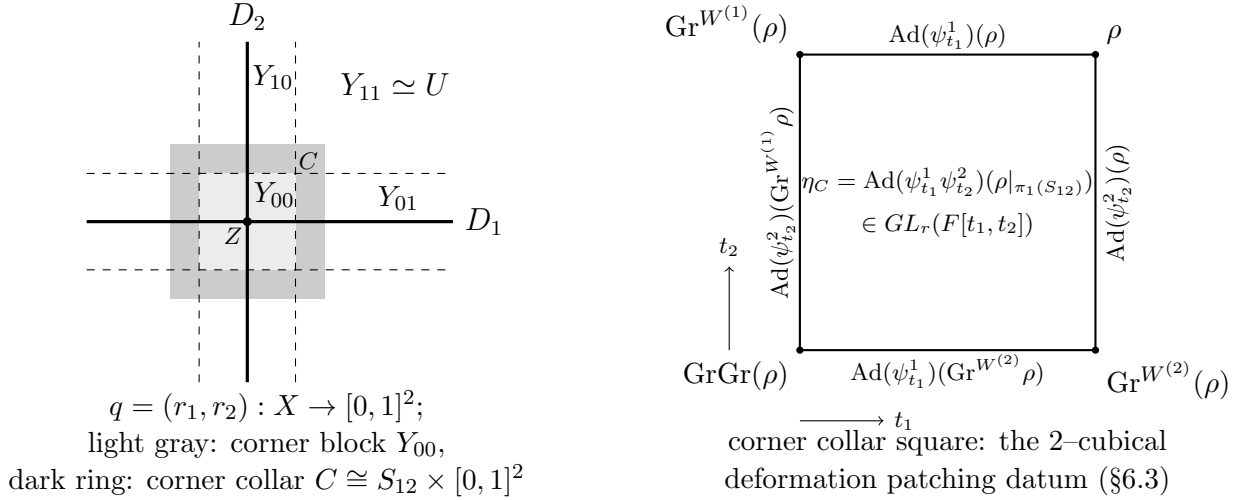

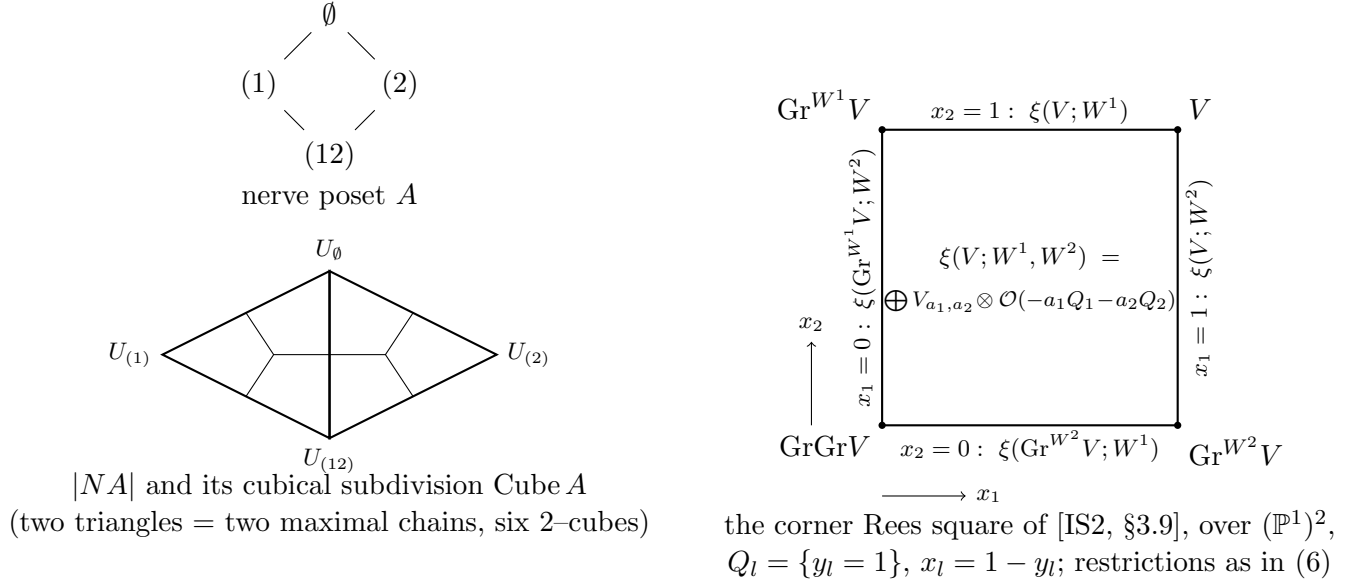
\begin{figure}[ht]
\centering
\begin{tikzpicture}[scale=0.85]
\begin{scope}[xshift=0cm,yshift=0.4cm]
\node (e) at (0,6.6) {\small $\emptyset$};
\node (d1) at (-1.1,5.5) {\small $(1)$};
\node (d2) at (1.1,5.5) {\small $(2)$};
\node (d12) at (0,4.4) {\small $(12)$};
\draw (e) -- (d1) -- (d12);
\draw (e) -- (d2) -- (d12);
\node at (0,3.75) {\small nerve poset $A$};
\begin{scope}[yshift=0cm]
\coordinate (E) at (0,2.6);
\coordinate (D12) at (0,0);
\coordinate (D1) at (-2.6,1.3);
\coordinate (D2) at (2.6,1.3);
\draw[thick] (E) -- (D1) -- (D12) -- cycle;
\draw[thick] (E) -- (D2) -- (D12) -- cycle;
\coordinate (mED1) at (-1.3,1.95);
\coordinate (mD1D12) at (-1.3,0.65);
\coordinate (mED12) at (0,1.3);
\coordinate (bL) at (-0.87,1.3);
\draw (mED1) -- (bL) -- (mD1D12);
\draw (bL) -- (mED12);
\coordinate (mED2) at (1.3,1.95);
\coordinate (mD2D12) at (1.3,0.65);
\coordinate (bR) at (0.87,1.3);
\draw (mED2) -- (bR) -- (mD2D12);
\draw (bR) -- (mED12);
\node[above] at (E) {\scriptsize $U_{\emptyset}$};
\node[left] at (D1) {\scriptsize $U_{(1)}$};
\node[right] at (D2) {\scriptsize $U_{(2)}$};
\node[below] at (D12) {\scriptsize $U_{(12)}$};
\node[align=center] at (0,-1.05) {\small $|NA|$ and its cubical
subdivision $\m{Cube}\,A$\\ \small (two triangles $=$ two maximal chains,
six $2$--cubes)};
\end{scope}
\end{scope}
\begin{scope}[xshift=8.6cm, yshift=0.6cm]
\draw[thick] (0,0) rectangle (4.6,4.6);
\fill (4.6,4.6) circle (1.6pt) node[above right] {\small $V$};
\fill (0,4.6) circle (1.6pt) node[above left] {\small $\Grr^{W^1}V$};
\fill (4.6,0) circle (1.6pt) node[below right] {\small $\Grr^{W^2}V$};
\fill (0,0) circle (1.6pt) node[below left] {\small $\Grr\Grr V$};
\node at (2.3,4.9) {\scriptsize $x_2=1:\ \xi(V;W^1)$};
\node at (2.3,-0.32) {\scriptsize $x_2=0:\ \xi(\Grr^{W^2}V;W^1)$};
\node[rotate=90] at (-0.34,2.3) {\scriptsize $x_1=0:\ \xi(\Grr^{W^1}V;W^2)$};
\node[rotate=90] at (4.94,2.3) {\scriptsize $x_1=1:\ \xi(V;W^2)$};
\node[align=center] at (2.3,2.3)
{\scriptsize $\xi(V;W^1,W^2)\;=$\\[2pt]
\tiny $\bigoplus V_{a_1,a_2}\!\otimes\cO(-a_1Q_1\!-\!a_2Q_2)$};
\draw[->] (0,-1.1) -- (1.3,-1.1) node[right] {\scriptsize $x_1$};
\draw[->] (-1.1,0) -- (-1.1,1.3) node[above] {\scriptsize $x_2$};
\node[align=center] at (2.3,-1.85)
{\small the corner Rees square of \cite[\S 3.9]{IS2}, over $(\p^1)^2$,\\
\small $Q_l=\{y_l=1\}$, $x_l=1-y_l$;
restrictions as in \eqref{reesrestrictions}};
\end{scope}
\end{tikzpicture}
\caption{The multi--Rees construction of \cite{IS2} for two components.
Left: the nerve poset $A$ of the adapted covering (only chain
intersections occur), its order complex, and the cubical barycentric
subdivision on which the global Rees bundle lives. Right: the local Rees
bundle over a corner $2$--cube; comparing with the right--hand square of
Figure \ref{fig:glued}, the dictionary is $x_i\leftrightarrow t_i$,
Rees frame $\leftrightarrow$ $\m{Ad}(\psi^1_{t_1}\psi^2_{t_2})$
(Proposition \ref{reesdictionary}).}
\label{fig:rees}
\end{figure}

\subsection{Comparison of the two proofs of triviality of the secondary
classes}
\label{cscomparison}

Both proofs follow Reznikov's skeleton: (i) construct the extended class
and show it lifts the Deligne Chern class of the canonical extension;
(ii) prove rigidity under deformations of $\rho$ inside the unipotent
locus, hence reduce to $\rho$ defined over a number field $F$ and to
proving, for every embedding $\sigma:F\hookrightarrow \comx$, that
(iii) the volume regulator vanishes when $\rho^{\sigma}$ underlies a
polarized complex VHS; then (iv) Mochizuki's deformation to a VHS and
Borel's theorem give torsion. Steps (ii) and (iv) are essentially
identical in the two papers (variational formula with triangular
integrand; Borel + Reznikov). The interesting comparisons are in (i) and
(iii).

\subsubsection{Step (i): lifting the Deligne Chern class.}
Here the class is the differential character of the patched connection
$\nabla^{\#}$ on $X$; since the Chern forms vanish identically (strict
triangularity for the total--degree filtration, Proposition
\ref{binildef}), the character lies in $H^{2p-1}(X,\comx/\Z)$, and the
comparison with the Deligne Chern class is the Dupont--Hain--Zucker
$F^1$--connection argument carried out directly on the projective $X$
(\S \ref{deligne}). In \cite{IS2} the class is carried by an
$F^1$--connection on the algebraic Rees bundle over the cubical
realization; the comparison with the Deligne--Beilinson Chern character is
done through the Burgos--Gil arithmetic Chern character on simplicial
schemes, the key computational fact being
$H^{2p}_{\cD}(\A^n,\Z(p))\cong\comx/\Z$ on each affine cube, so that the
\v{C}ech--to--Deligne spectral sequence identifies
$H^{2p}_{\cD}$ of the affine realization with
$H^{2p-1}(X,\comx/\Z)$. The two mechanisms agree where they overlap:
restricted to a single cube, the Burgos--Gil/$F^1$ comparison reduces to
the same DHZ lemma (independence of the class of the choice of
$F^1$--connection) that we use on $X$.

\subsubsection{Step (i$'$): the $K$--theory carrier.}
On our side the homotopy--theoretic carrier is the $2$--cubical pushout
$BGL(F)^+_{{\rm def},2}$ with $BGL(F[t_1,t_2])^+$ at the $2$--face, and the
engine making the vertex inclusion a homology equivalence is Quillen's
homotopy invariance $K_*(F)\cong K_*(F[t_1,t_2])$ applied twice (Lemma
\ref{cubicaldefisom}). In \cite{IS2} the carrier is the $\A^1$--homotopy
type of the affine cubical realization: a vector bundle on the cubical
realization classifies a map to $BGL(F)^+$ because each affine cube
$\A^n_F$ is $\A^1$--contractible. These are the same principle in two
costumes --- $\A^1$--contractibility of $\A^n$ \emph{is} homotopy
invariance of $K$--theory in the form used --- and Proposition
\ref{reesdictionary} makes the translation concrete: our deformation
triple $(\eta_{ij};\eta_{\rm edges};\eta_C)$ is literally the system of
monodromies of the Rees local system on the cubes.

\subsubsection{Step (iii): vanishing of the volume regulators for a VHS.}
This is where the two proofs genuinely diverge, although they consume the
same two inputs (the flat polarization of the VHS, and commuting
$Q$--self--adjoint logarithms $N_1,N_2$) and the same Lie--theoretic fact
(odd products in $\mathfrak{u}(p,q)$ have purely imaginary trace;
equivalently, there are no invariant polynomials on
$\mathfrak{su}(p,q)$ detecting the Borel classes).

\emph{Hermitian $K$--theory route (this note, \S \ref{hermitian}).} The
self--dual splitting of Proposition \ref{CKSsplitting} (for which the CKS
several--variable $SL_2$--orbit theorem produces distinguished
representatives) is a bigrading
$V=\bigoplus V_{a,b}$ splitting both weight filtrations
\emph{compatibly with the polarization}
($\langle V_{a,b},V_{c,d}\rangle =0$ unless $(c,d)=(-a,-b)$); the scalings
$T_{t_1,t_2}$ therefore preserve the form, the two--variable deformation
is a family in $O_{p,q}(\comx[t_1,t_2])$, and the classifying map factors
through $BO_{\infty,\infty}(\comx)^+_{{\rm def},2}$. Karoubi's homotopy
invariance of $L$--theory away from $2$, applied twice, identifies this
space rationally with $BO_{\infty,\infty}(\comx)^+$, and Reznikov's
invariant--theory computation kills the pullback of the Borel classes
\emph{once and for all on the universal space}. The vanishing of
$Vol_{2p-1}(\rho/X)$ is then formal for every $X$ mapping in.

\emph{Differential--geometric route (\cite{IS2}, \S\S 9--12 there).} The
polarization is used to build, on $X$ itself, an \emph{indefinite
hermitian metric $\tilde g$ on the canonical extension} together with a
$\tilde g$--preserving connection compatible with the patching data. The
existence is a genuine global problem, solved in two moves: the poset of
$N_{\bullet}$--isotropic filtrations on a fiber is contractible (proved
with Quillen's Theorem A), and the \v{C}ech section theorem converts
fiberwise contractibility into a global compatible choice over the adapted
covering; averaging by partitions of unity then produces $(\tilde
g,\tilde\nabla)$. The vanishing is now \emph{pointwise}: along a path of
$\tilde g$--preserving connections, $\dot\nabla_t$ and the curvature take
values in $\mathfrak{u}(\tilde g)$, so
$\Re\,\m{Tr}(\dot\nabla_t\wedge F^{p-1})=0$ as a differential form, and
the volume regulator vanishes at the level of forms, not merely of
cohomology classes.

The relation between the two: our self--dual bigrading, together with its
polarization compatibility, is exactly a (distinguished, globally
constant--along--the--collar) point of the contractible space of
$N_{\bullet}$--isotropic data of \cite{IS2}; conversely, an
$N_{\bullet}$--isotropic filtration is what remains of the bigrading
when one only retains the interaction with the form. One can say that
\cite{IS2} \emph{globalizes the choice} (\v{C}ech section over $X$) and
then kills the form pointwise, while we \emph{rigidify the choice}
(one bigrading at the corner) and kill the class universally on
$BO_{\infty,\infty}$. For $k=2$ the two arguments are interchangeable.

\subsubsection{Coincidence of the classes.}
\begin{proposition}\label{sameclass}
Let $\rho:\pi_1(X-D)\rar GL_r(\comx)$ be unipotent at infinity along
$D=D_1\cup D_2$ as in Theorem \ref{mainthm}. The extended class
$\what{c}_p(\rho/X)$ of Corollary \ref{bipatchingsame} coincides with the
class $CS_p(\nabla^{\rm Del})$ of \cite{IS2} specialized to $k=2$.
\end{proposition}

\begin{proof}[Proof sketch]
Both classes are pullbacks of the universal regulator class in
$H^{2p-1}(BGL(F)^+,\comx/\Z)$ under classifying maps determined by
degeneration data for $\rho$: ours via
$\xi_{\rho}:X\rar BGL(F)^+_{{\rm def},2}$ built from
$(\rho,\Wone,\Wtwo)$ and a simultaneous splitting (Corollary
\ref{bipatchingsame}); that of \cite{IS2} via the Rees bundle on the
cubical realization built from the \v{C}ech section
$\{W(N_I),\tau(I)\}$ and a locally abelian splitting (Theorems 7.5 and
8.7 of \cite{IS2}). By Proposition \ref{reesdictionary} the first datum is
itself of the second kind (a strict collection of patching data over the
covering by blocks, with the bifiltration providing the locally abelian
structure on the corner cube), and the two data lie in the same
contractible space of choices: the poset of patching data is contractible
fiberwise (\cite[Theorem 5.10]{IS2} together with the two--filtration
splitting lemma), and the uniqueness theorem for bundles with prescribed
patching functor (\cite[Theorem 7.15]{IS2}) identifies the two bundles on
the cubical realization together with their patching data. Hence the two
classifying maps are homotopic and the pulled--back classes agree. Both
classes also lift the same Deligne Chern class
$c_p^{\cD}(\ov E)$ (Proposition \ref{bideligne} here;
Burgos--Gil comparison in \cite{IS2}), which pins down the common class up
to the image of $H^{2p-1}(X,\comx)$, removed by the rigidity statements on
both sides. A full verification of the homotopy between the classifying
maps requires the bookkeeping of \cite[\S\S 7--8]{IS2} restricted to
$k=2$, which we have not rewritten here.
\end{proof}


\appendix
\section{Detailed verifications}
\label{appendixproofs}

This appendix writes out the verifications that earlier versions of this
note only sketched: the global construction of simultaneous splittings at
the corner (\S \ref{app:splittings}, used in Theorem \ref{biconstruction});
the automatic bi--graded--extendability of the intermediate filtrations in
the Jordan--H\"older string (\S \ref{app:string}, used in Proposition
\ref{biindepfilt}); the semicontinuity and transport lemmas for families of
pairs of filtrations (\S \ref{app:rigidity}, used in Proposition
\ref{birigidity}); the self--dual splitting of a pair of weight filtrations
(\S \ref{app:selfdual}, which is the corrected form of Proposition
\ref{CKSsplitting}); and the cubical Mayer--Vietoris bookkeeping
(\S \ref{app:cubical}, used in Lemmas \ref{cubicaldefisom} and
\ref{bicompare1}).

\subsection{Simultaneous splittings at the corner}
\label{app:splittings}

\begin{lemma}[Corner normal form and strictness]\label{strictcorner}
Let $V_{12}$ be the polydisc--bundle neighborhood of $Z$, $L$ the local
system on $V_{12}^{\ast}$, and $\ov E|_{V_{12}}$ the canonical extension.
Then locally on $V_{12}$ there are isomorphisms
$\ov E \cong \cO\otimes V$ ($V$ a fiber of $L$) under which every
sub--local system $W\subset L$ corresponds to the constant subbundle
$\cO\otimes W$. Consequently, for any pair $(\Wone,\Wtwo)$ of filtrations
of $L|_{V_{12}^{\ast}}$ by sub--local systems, all the subsheaves
$$
\Wone_a\cap\Wtwo_b, \qquad
\Wone_{a-1}\cap\Wtwo_b+\Wone_a\cap\Wtwo_{b-1}
\;\subset\; \ov E|_{V_{12}}
$$
(canonical extensions of the corresponding sub--local systems, resp.\ their
sums) are strict $\mathcal{C}^{\infty}$ subbundles of constant rank, and
$$
\big(\Wone_{a-1}\cap\Wtwo_b\big)\cap\big(\Wone_a\cap\Wtwo_{b-1}\big)
=\Wone_{a-1}\cap\Wtwo_{b-1},
$$
so that
${\rm rk}\big(\Wone_{a-1}\cap\Wtwo_b+\Wone_a\cap\Wtwo_{b-1}\big)
= d_{a-1,b}+d_{a,b-1}-d_{a-1,b-1}$ with
$d_{a,b}:={\rm rk}(\Wone_a\cap\Wtwo_b)$.
\end{lemma}

\begin{proof}
Work locally near a corner point, with coordinates in which
$D_i=\{z_i=0\}$, and let $v$ run through a multivalued flat frame of $L$.
The canonical extension is framed by the single--valued sections
$$
\tilde v \;:=\;
\exp\!\Big(-\frac{\log z_1}{2\pi i}\,\widetilde N_1
-\frac{\log z_2}{2\pi i}\,\widetilde N_2\Big)\, v ,
\qquad \widetilde N_i := \log\rho(\gamma_i),
$$
(the twist operator undoes the monodromy; this is Deligne's local
description \cite{De}). If $W\subset V$ is stable under the full local
monodromy group, it is in particular stable under $\widetilde N_1$ and
$\widetilde N_2$ (polynomials in $\rho(\gamma_i)$), hence stable under the
twist operator; therefore the span of the sections $\tilde w$, $w\in W$, is
the constant subbundle $\cO\otimes W$ in the frame $\{\tilde v\}$. Every
sub--local system over $V_{12}^{\ast}$ is monodromy stable, so all the
subsheaves in the statement are, in this frame, the constant subbundles
attached to the corresponding subspaces
$\Wone_a\cap\Wtwo_b\subset V$, etc.; in particular they are strict of
constant rank, and all rank identities reduce to linear algebra in the
single fiber $V$. There the displayed identity is elementary,
$$
(\Wone_{a-1}\cap\Wtwo_b)\cap(\Wone_a\cap\Wtwo_{b-1})
=\Wone_{a-1}\cap\Wone_a\cap\Wtwo_b\cap\Wtwo_{b-1}
=\Wone_{a-1}\cap\Wtwo_{b-1},
$$
and the rank of the sum is
$\dim A+\dim B-\dim(A\cap B)$. The local normal forms glue: the subbundles
themselves are globally defined (canonical extensions of globally defined
sub--local systems), and strictness is a local statement.
\end{proof}

\begin{lemma}[Metric bigrading]\label{metricsplit}
Let $E$ be a $\mathcal{C}^{\infty}$ complex vector bundle on a manifold
$M$, and let $(\Wone,\Wtwo)$ be two filtrations of $E$ by strict subbundles
such that all intersections $\Wone_a\cap\Wtwo_b$ and all sums
$\Wone_{a-1}\cap\Wtwo_b+\Wone_a\cap\Wtwo_{b-1}$ are strict subbundles.
Let $h$ be a hermitian metric on $E$ and define
$$
E_{a,b} \;:=\;
\big(\Wone_a\cap\Wtwo_b\big)\ \cap\
\Big(\big(\Wone_{a-1}\cap\Wtwo_b\big)+\big(\Wone_a\cap\Wtwo_{b-1}\big)
\Big)^{\perp_h}.
$$
Then $E=\bigoplus_{a,b}E_{a,b}$ is a $\mathcal{C}^{\infty}$ bigrading
splitting both filtrations,
$$
\Wone_a=\bigoplus_{a'\leq a}E_{a',\bullet},
\qquad
\Wtwo_b=\bigoplus_{b'\leq b}E_{\bullet,b'} ,
$$
and it depends smoothly on $(h,\Wone,\Wtwo)$; in particular a smooth
(resp.\ real--analytic) family of such data yields a smooth (resp.\
real--analytic) family of bigradings.
\end{lemma}

\begin{proof}
The statement is pointwise (smoothness follows because intersections and
$h$--orthocomplements of strict subbundles with constant--rank outcome are
smooth operations in local frames, and the hypotheses fix all ranks). Fix a
fiber and write $V$ for it, and set
$$
V_{a,b}:=\Wone_a\cap\Wtwo_b,
\qquad
S_{a,b}:=V_{a-1,b}+V_{a,b-1}\subset V_{a,b},
$$
so that
$V_{a,b}=E_{a,b}\oplus S_{a,b}$ (orthogonal decomposition inside
$V_{a,b}$).

\emph{Spanning.} We prove
$V_{a,b}=\sum_{a'\leq a,\ b'\leq b}E_{a',b'}$ by induction on $a+b$
(the filtrations are bounded below, so the induction starts). Indeed
$V_{a,b}=E_{a,b}+S_{a,b}$ and, by the inductive hypothesis applied to
$V_{a-1,b}$ and $V_{a,b-1}$,
$S_{a,b}=\sum_{a'\leq a-1,\,b'\leq b}E_{a',b'}
+\sum_{a'\leq a,\,b'\leq b-1}E_{a',b'}$, which together with $E_{a,b}$
spans the claimed sum. Taking $a,b\gg 0$ gives $V=\sum E_{a,b}$.

\emph{Dimension count and directness.} By Lemma \ref{strictcorner}
(or by the same one--line computation in $V$),
$V_{a-1,b}\cap V_{a,b-1}=V_{a-1,b-1}$, so
$$
\dim E_{a,b}=d_{a,b}-\dim S_{a,b}
= d_{a,b}-d_{a-1,b}-d_{a,b-1}+d_{a-1,b-1} .
$$
Summing over all $(a,b)$, the right side telescopes in each variable to
$\dim V$. Hence $\sum_{a,b}\dim E_{a,b}=\dim V$; combined with the spanning
property, $V=\bigoplus_{a,b}E_{a,b}$.

\emph{Splitting the filtrations.} $E_{a',b'}\subset\Wone_{a'}$, so
$\bigoplus_{a'\leq a}E_{a',\bullet}\subset\Wone_a$; the dimensions agree
because $\sum_{a'\leq a,\,b'}\dim E_{a',b'}$ telescopes to
$\lim_{b\to\infty}d_{a,b}={\rm rk}\,\Wone_a$. Hence equality, and
symmetrically for $\Wtwo$. (This is Deligne's two--filtration argument
\cite[1.2.8]{De2}, with the splitting made canonical by the metric.)
\end{proof}

\subsection{Bi--graded--extendability along the Jordan--H\"older string}
\label{app:string}

\begin{lemma}\label{autoE23}
Let $W$ be a filtration of $L|_{(B_1\cup B_{12})^{\ast}}$ by sub--local
systems satisfying condition (E1) of Definition \ref{bigradedext}
(the graded extends across $D_1$), and let $\Wtwo$ be a filtration of
$L|_{(B_2\cup B_{12})^{\ast}}$ by sub--local systems satisfying its (E1)
across $D_2$. Then the pair $(W,\Wtwo)$ automatically satisfies (E2) and
(E3): it is a bi--graded--extendable pair.
\end{lemma}

\begin{proof}
All verifications take place over $V_{12}^{\ast}$, whose fundamental group
contains the loops $\gamma_1,\gamma_2$ and surjects onto
$\pi_1(V_{12}-D_2)$ (killing $\gamma_1$) and onto $\pi_1(V_{12})$
(killing both); extensions of local systems across the branches exist iff
the corresponding monodromy factors through these quotients, and are then
unique.

First note that $\rho(\gamma_2)$, hence $N_2=\log\rho(\gamma_2)$,
preserves every step $W_j$: indeed $\gamma_2$ lies in the image of
$\pi_1(V_{12}^{\ast})\rar\pi_1\big((B_1\cup B_{12})^{\ast}\big)$ and $W_j$
is a sub--local system over the larger open set, i.e.\ a subspace stable
under the whole $\pi_1\big((B_1\cup B_{12})^{\ast}\big)$--action.
Symmetrically $\rho(\gamma_1)$ and $N_1$ preserve every $\Wtwo_b$.
Also, condition (E1) for $W$ says exactly that $\gamma_1$ acts trivially on
each graded piece:
\begin{equation}\label{gamma1drops}
\big(\rho(\gamma_1)-1\big)\,W_j \subset W_{j-1} ,
\end{equation}
and (E1) for $\Wtwo$ gives symmetrically
$(\rho(\gamma_2)-1)\Wtwo_b\subset\Wtwo_{b-1}$.

\emph{(E2).} On $W_j\cap\Wtwo_b$ we compute, using \eqref{gamma1drops} and
the $\rho(\gamma_1)$--stability of $\Wtwo_b$:
$$
\big(\rho(\gamma_1)-1\big)\big(W_j\cap\Wtwo_b\big)
\;\subset\; W_{j-1}\cap\Wtwo_b ,
$$
a single summand of
$W_{j-1}\cap\Wtwo_b+W_j\cap\Wtwo_{b-1}$. Hence $\gamma_1$ acts trivially
on the double graded
$\Grr\Grr(L)=\bigoplus (W_j\cap\Wtwo_b)/(W_{j-1}\cap\Wtwo_b
+W_j\cap\Wtwo_{b-1})$; symmetrically for $\gamma_2$. Since
$\pi_1(V_{12})$ is the quotient of $\pi_1(V_{12}^{\ast})$ by the normal
subgroup generated by $\gamma_1,\gamma_2$ (a product of disc bundles over
$Z$), the double graded extends to a local system on $V_{12}$.

\emph{(E3).} The extension of $\Grr^{W}(L)$ across $D_1$ over $V_{12}-D_2$
exists by (E1) and carries the $\pi_1(V_{12}-D_2)$--action through which
the $\pi_1(V_{12}^{\ast})$--action factors. The filtration induced by
$\Wtwo$ on $\Grr^{W}(L)$, with steps the images of $W_j\cap\Wtwo_b$,
is stable under all of $\pi_1(V_{12}^{\ast})$ (each $W_j\cap\Wtwo_b$ is an
intersection of stable subspaces), hence under the quotient
$\pi_1(V_{12}-D_2)$: it is a filtration of the extended local system by
sub--local systems. Symmetrically on $\Grr^{\Wtwo}$.
\end{proof}

\subsection{Families of pairs: semicontinuity, limits, transport}
\label{app:rigidity}

\begin{lemma}[Subdivision and limit pairs]\label{rigiditylimits}
Let $\rho(s)$, $s\in[0,1]$, be a real--algebraic family of representations
of $\pi_1(U)$ with unipotent $\rho(s)(\gamma_1),\rho(s)(\gamma_2)$, and let
$\Wone(s),\Wtwo(s)$ denote the corresponding kernel filtrations of
$N_1(s),N_2(s)$. Then:
\begin{itemize}
\item[(a)] There is a finite subdivision
$0=s_0<s_1<\cdots<s_m=1$ such that on each open subinterval
$(s_{i-1},s_i)$ \emph{all} the integers
$$
{\rm rk}\,\Wone_a(s),\quad {\rm rk}\,\Wtwo_b(s),\quad
d_{a,b}(s)={\rm rk}\big(\Wone_a(s)\cap\Wtwo_b(s)\big)
$$
are constant, and the subspaces vary real--analytically (indeed
real--algebraically) in the Grassmannians of the fiber $V$.
\item[(b)] At each endpoint $s_i$, one--sided limits
$\ov W{}^{(1)}_a,\ov W{}^{(2)}_b$ of the filtrations exist in the
Grassmannians. They are filtrations of $L(s_i)$ by sub--local systems
satisfying condition (E1); hence by Lemma \ref{autoE23} the limit pair
$(\ov W{}^{(1)},\ov W{}^{(2)})$ is bi--graded--extendable for
$\rho(s_i)$.
\item[(c)] Fix a metric $h$ on the underlying $\mathcal{C}^{\infty}$ bundle,
constant in $s$. On each closed subinterval $[s_{i-1},s_i]$, after a
one--sided reparametrization $s=s_{i-1}+u^{N}$ (resp.\ $s=s_i-u^{N}$)
near the endpoints, the family of metric bigradings $E_{a,b}(s)$ of Lemma
\ref{metricsplit} (formed from the pair $\Wone(s),\Wtwo(s)$ on the open
part and from the limit pair at the endpoints) is real--analytic up to and
including the endpoints, with constant ranks.
\end{itemize}
\end{lemma}

\begin{proof}
(a) Each subspace in question is the kernel of a polynomial family of
matrices: $\Wone_a(s)\cap\Wtwo_b(s)=\ker\binom{N_1(s)^a}{N_2(s)^b}$
(stacked matrix). For a real--algebraic one--parameter family, the rank of
a matrix is constant off a finite set (the locus where all minors of the
generic size vanish); take the union of these finite sets over the finitely
many $(a,b)$. On the complement, the kernel of a constant--rank algebraic
family of matrices is an algebraic (hence real--analytic) map to the
Grassmannian.

(b) A real--algebraic arc in a projective variety (the Grassmannian)
extends across a puncture: parametrize the germ at $s_i$ by a one--sided
Puiseux series (equivalently, apply the valuative criterion of properness
to the graph closure), so the limits $\ov W{}^{(1)}_a$ exist; inclusions
$\Wone_{a-1}(s)\subset\Wone_a(s)$ and monodromy stability
$\rho(s)(g)\Wone_a(s)=\Wone_a(s)$ are closed conditions on the pair
(subspace, parameter), so they pass to the limit and the limits form
filtrations by sub--local systems of $L(s_i)$ (recall the underlying
$\mathcal{C}^{\infty}$ local system structure is trivialized along the
family, so the limit subspaces are acted on by $\pi_1$ through
$\rho(s_i)$). Condition (E1), in the form \eqref{gamma1drops}
$(\rho(s)(\gamma_1)-1)\Wone_a(s)\subset\Wone_{a-1}(s)$, holds on the open
interval (kernel filtrations satisfy it, \cite[Lemma 2.6]{IS}) and is again
a closed condition on constant--rank families, hence holds in the limit.
Lemma \ref{autoE23} applies.

(c) After the reparametrization, all the algebraic subspace families of
(a) become real--analytic on the \emph{closed} interval in the variable
$u$ (a Puiseux parametrization of a real--algebraic arc is analytic at the
endpoint). The subtle point is that the ranks $d_{a,b}$ could a priori jump
at the endpoint, i.e.\ the limit of $\Wone_a(s)\cap\Wtwo_b(s)$ could be
smaller than $\ov W{}^{(1)}_a\cap\ov W{}^{(2)}_b$. To see that no jump
occurs, consider the family of $h$--orthogonal projectors
$P_{a,b}(u)$ onto the subspaces $E_{a,b}(u)$ formed on the open part; each
$P_{a,b}(u)$ is a real--analytic family of matrices (algebraic operations
on constant--rank analytic families), and it extends analytically to
$u=0$ because its entries are bounded (projectors of norm bounded in terms
of the angles between complementary analytic subspace families; after
shrinking, a real--analytic family of complementary subspaces on the
half--open interval with entries meromorphic at $0$ has projectors
meromorphic at $0$, and boundedness excludes poles: orthogonal projectors
have operator norm $1$). The rank of the limit projector equals its trace,
which is constant along the family; hence the limits
$\ov E_{a,b}:=\mathrm{im}\,P_{a,b}(0)$ form a decomposition of $V$ with the
same dimensions, and being a splitting of the pair of filtrations is a
closed condition, so $\ov E_{a,b}$ splits the limit pair
$(\ov W{}^{(1)},\ov W{}^{(2)})$:
$$
\ov W{}^{(1)}_a=\bigoplus_{a'\leq a}\ov E_{a',\bullet},\qquad
\ov W{}^{(2)}_b=\bigoplus_{b'\leq b}\ov E_{\bullet,b'} .
$$
But intersections of two coordinate subspaces of one and the same
bigrading are computed componentwise:
$\ov W{}^{(1)}_a\cap\ov W{}^{(2)}_b
=\bigoplus_{a'\leq a,\,b'\leq b}\ov E_{a',b'}$, whose dimension is the
generic $d_{a,b}$. So no intersection rank jumps at the endpoint, the
strictness hypotheses of Lemma \ref{metricsplit} hold on the closed
interval, and the bigrading family (which agrees with the limit
construction at $u=0$ up to the $h$--orthogonalization, itself analytic) is
real--analytic on the closed interval.
\end{proof}

\begin{lemma}[Transport of a family of decompositions]\label{katolemma}
Let $\{P_k(s)\}_{k=1,\ldots,m}$, $s\in[0,1]$, be a $\mathcal{C}^1$ family
of complementary projectors on a finite--dimensional vector space
($\sum_kP_k=1$, $P_kP_l=\delta_{kl}P_k$), or a family of such on a
$\mathcal{C}^{\infty}$ bundle. Put
$$
A(s) \;:=\; \tfrac12\sum_{j}\big[\,P_j'(s),\,P_j(s)\,\big] .
$$
Then the solution of $U'(s)=A(s)U(s)$, $U(0)=1$, satisfies
$P_k(s)\,U(s)=U(s)\,P_k(0)$ for all $k$ and $s$: the gauge transformation
$U(s)$ carries the constant decomposition onto the moving one.
\end{lemma}

\begin{proof}[Proof (Kato \cite{Kato})]
It suffices to prove the commutator identity
\begin{equation}\label{katoidentity}
\Big[\sum_j[P_j',P_j],\,P_k\Big] \;=\; 2\,P_k' ,
\end{equation}
for then $(U^{-1}P_kU)'=U^{-1}\big(P_k'-[A,P_k]\big)U=0$.
Write $Q:=\sum_jP_jP_j'$ and $R:=\sum_jP_j'P_j$; differentiating
$\sum_jP_j=1$ gives $\sum_jP_j'=0$, hence $Q+R=0$. Then
$$
\Big[\sum_j[P_j',P_j],P_k\Big]
= P_k'P_k+P_kP_k' - QP_k + P_kQ
= P_k' + [P_k,Q] .
$$
Evaluate blockwise: for indices $a,b$, differentiating $P_aP_b=\delta_{ab}P_a$
yields $P_a'P_b=-P_aP_b'$ for $a\neq b$ and $P_aP_a'P_a=0$; one computes
$P_a\,Q\,P_b=P_aP_a'P_b$, hence
$P_a[P_k,Q]P_b=(\delta_{ak}-\delta_{bk})\,P_aP_a'P_b$. Comparing with the
blocks of $P_k'$ (which satisfy $P_aP_k'P_b=0$ unless exactly one of $a,b$
equals $k$, and $P_kP_k'P_b$, $P_aP_k'P_k$ in the two mixed cases, with
$P_aP_a'P_k=P_aP_k'P_k$ from $P_a'P_k=-P_aP_k'$) gives
$P_a(P_k'+[P_k,Q])P_b = 2\,P_aP_k'P_b$ in every block, which is
\eqref{katoidentity}.
\end{proof}

In the proof of Proposition \ref{birigidity}, Lemmas \ref{rigiditylimits}
and \ref{katolemma} are combined as follows: on each closed subinterval the
projector family of the metric bigradings is $\mathcal{C}^1$ (after the
harmless one--sided reparametrization, which does not affect the endpoint
classes), and the transport $U(s)$ identifies the moving pair of
filtrations, hence the whole moving localized patching collection, with the
constant one at the left endpoint; the variational vanishing argument then
proceeds with constant filtrations exactly as in \cite[Lemma 6.5]{IS}.

\subsection{The self--dual splitting: proof of Proposition
\ref{CKSsplitting}}
\label{app:selfdual}

Throughout, $\langle\cdot,\cdot\rangle$ is a nondegenerate hermitian form
on the finite--dimensional complex vector space $V$, preserved by the image
of $\rho|_{\pi_1(V_{12}^{\ast})}$, and $N_1,N_2$ are the commuting monodromy
logarithms. Write $A^{\dagger}$ for the $\langle\cdot,\cdot\rangle$--adjoint.

\begin{lemma}\label{Nisotropic}
$N_i^{\dagger}=-N_i$, i.e.\
$\langle N_iu,v\rangle+\langle u,N_iv\rangle=0$ for all $u,v$.
\end{lemma}
\begin{proof}
For real $t$ the function
$t\mapsto\langle e^{tN_i}u,\,e^{tN_i}v\rangle$ is a polynomial in $t$:
each entry of $e^{tN_i}$ is a polynomial in $t$, and since $t$ is real the
conjugated entries appearing through the second (conjugate--linear) slot
are again polynomials in $t$. It takes the value
$\langle u,v\rangle$ at every integer $t$, because
$e^{nN_i}=\rho(\gamma_i)^n$ preserves the form; a polynomial with
infinitely many coincidences is constant. Differentiating at $t=0$ gives
the claim.
\end{proof}

\begin{lemma}[Self--duality of the weight filtration]\label{selfdualW}
$W(N_i)_a^{\perp}=W(N_i)_{-a-1}$ for all $a$.
\end{lemma}
\begin{proof}
Write $N=N_i$, $W=W(N)$, and define $W'_a:=(W_{-a-1})^{\perp}$, an
increasing filtration. For $v\in W'_a$ and $u\in W_{-a+1}$ we have, by
Lemma \ref{Nisotropic},
$\langle Nv,u\rangle=-\langle v,Nu\rangle=0$ since
$Nu\in W_{-a-1}$; thus $NW'_a\subset(W_{-a+1})^{\perp}=W'_{a-2}$. On the
graded pieces, $\Grr^{W'}_a=(W_{-a-1})^{\perp}/(W_{-a})^{\perp}$ is
canonically the conjugate dual of $\Grr^W_{-a}$, and the endomorphism
induced by $N$ corresponds under this identification to $-$ (conjugate
dual of $N$); hence
$N^k:\Grr^{W'}_k\rar\Grr^{W'}_{-k}$ is, up to sign, the conjugate dual of
$N^k:\Grr^W_{k}\rar\Grr^W_{-k}$, an isomorphism. By uniqueness of the
monodromy weight filtration (both properties characterize it), $W'=W$.
\end{proof}

\begin{lemma}[Perfect pairing on the double graded]\label{grgrpairing}
Write $\Grr\Grr_{a,b}
:=(W^1_a\cap W^2_b)\big/\big(W^1_{a-1}\cap W^2_b+W^1_a\cap W^2_{b-1}\big)$
with $W^i:=W(N_i)$. The form induces a well--defined nondegenerate pairing
$$
\Grr\Grr_{a,b}\ \times\ \Grr\Grr_{-a,-b}\ \lrar\ \comx .
$$
In particular $\dim\Grr\Grr_{a,b}=\dim\Grr\Grr_{-a,-b}$.
\end{lemma}
\begin{proof}
\emph{Well--defined.} By Lemma \ref{selfdualW} one has
$\langle W^1_{a-1},\,W^1_{-a}\rangle=0$
(since $W^1_{-a}=(W^1_{a-1})^{\perp}$), and likewise
$\langle W^2_{b-1},\,W^2_{-b}\rangle=0$; hence both summands of the
denominator at $(a,b)$ pair to zero with $W^1_{-a}\cap W^2_{-b}$, and
symmetrically.

\emph{Nondegenerate.} Let $v\in W^1_a\cap W^2_b$ pair to zero with all of
$W^1_{-a}\cap W^2_{-b}$. Then
$v\in(W^1_{-a}\cap W^2_{-b})^{\perp}
=(W^1_{-a})^{\perp}+(W^2_{-b})^{\perp}
=W^1_{a-1}+W^2_{b-1}$ (using $(A\cap B)^{\perp}=A^{\perp}+B^{\perp}$ for a
nondegenerate form and Lemma \ref{selfdualW}). It remains to prove the
modular identity
\begin{equation}\label{modular}
\big(W^1_a\cap W^2_b\big)\cap\big(W^1_{a-1}+W^2_{b-1}\big)
= W^1_{a-1}\cap W^2_b\ +\ W^1_a\cap W^2_{b-1} .
\end{equation}
The inclusion $\supseteq$ is clear. For $\subseteq$, choose \emph{any}
bigrading $V=\bigoplus V_{c,d}$ splitting the pair $(W^1,W^2)$ (one exists
by Lemma \ref{metricsplit} applied to the single fiber; note this uses no
form, so there is no circularity). All four subspaces in \eqref{modular}
are sums of coordinate pieces $V_{c,d}$ over the index sets
$$
\{c\leq a,\,d\leq b\}\cap\big(\{c\leq a-1\}\cup\{d\leq b-1\}\big)
=\{c\leq a-1,\,d\leq b\}\cup\{c\leq a,\,d\leq b-1\},
$$
which is exactly the index set of the right side. Both sides being
coordinate subspaces with the same index set, they are equal.
\end{proof}

\begin{proof}[Proof of Proposition \ref{CKSsplitting}]
Let $\mathcal S$ denote the set of bigradings $s=\{V_{a,b}\}$ splitting the
pair $(W^1,W^2)$; it is non--empty by Lemma \ref{metricsplit}. Let $G$ be
the group of automorphisms of $V$ preserving both filtrations and inducing
the identity on $\Grr\Grr$. Then:

\emph{$\mathcal S$ is a $G$--torsor.} Given $s,s'$, the compositions
$V_{a,b}\cong\Grr\Grr_{a,b}\cong V'_{a,b}$ assemble into $g\in G$ with
$g\cdot s=s'$; and $g\in G$ fixing $s$ acts on each $V_{a,b}$ as the map
induced on $\Grr\Grr_{a,b}$, i.e.\ as the identity.

\emph{$G$ is unipotent, with unique square roots.} For $g\in G$ and any
$s\in\mathcal S$, $(g-1)$ maps $V_{a,b}$ into
$\bigoplus_{a'\leq a,\,b'\leq b,\,(a',b')\neq(a,b)}V_{a',b'}$ (it preserves
$W^1_a\cap W^2_b$ and induces $0$ on $\Grr\Grr$), so $g-1$ strictly lowers
the total degree $a+b$ and is nilpotent. Hence
$h:=g^{1/2}:=\exp(\tfrac12\log g)$ is defined by the finite series, lies in
$G$ (the conditions defining $G$ are stable under polynomials in $g$
congruent to $1$), and is the unique square root of $g$ in $G$.

\emph{The duality involution.} For $s=\{V_{a,b}\}\in\mathcal S$ set
$$
V^{\sigma}_{a,b}
:=\Big(\bigoplus_{(c,d)\neq(-a,-b)}V_{c,d}\Big)^{\perp} .
$$
We claim $\sigma(s):=\{V^{\sigma}_{a,b}\}\in\mathcal S$.
(1) $\dim V^{\sigma}_{a,b}=\dim V_{-a,-b}$ (orthogonal of a subspace of
complementary dimension) $=\dim\Grr\Grr_{-a,-b}=\dim\Grr\Grr_{a,b}$ by
Lemma \ref{grgrpairing}.
(2) $V^{\sigma}_{a,b}\subset W^1_a\cap W^2_b$: indeed
$W^1_{-a-1}=\bigoplus_{c\leq -a-1}V_{c,\bullet}$ is contained in the
subspace being orthogonalized (all its indices have $c\leq -a-1<-a$), so
$V^{\sigma}_{a,b}\subset(W^1_{-a-1})^{\perp}=W^1_a$ by Lemma
\ref{selfdualW}; symmetrically for $W^2_b$.
(3) The $V^{\sigma}_{a,b}$ are independent: if $\sum v_{a,b}=0$ with
$v_{a,b}\in V^{\sigma}_{a,b}$, pair with $V_{-a_0,-b_0}$: all terms with
$(a,b)\neq(a_0,b_0)$ vanish by definition of $\sigma$, so
$\langle v_{a_0,b_0},V_{-a_0,-b_0}\rangle=0$; since also
$\langle v_{a_0,b_0},V_{c,d}\rangle =0$ for $(c,d)\neq(-a_0,-b_0)$, we get
$v_{a_0,b_0}\perp V$, hence $v_{a_0,b_0}=0$.
By (1)+(3), $V=\bigoplus V^{\sigma}_{a,b}$, and by (2) plus the dimension
count as in Lemma \ref{metricsplit}, $\sigma(s)$ splits both filtrations.
Moreover $\sigma(s)$ is characterized by:
$\langle V^{\sigma}_{a,b},V_{c,d}\rangle=0$ unless $(c,d)=(-a,-b)$; since
this orthogonality relation is (conjugate--)symmetric, applying it twice
gives $\sigma^2=\mathrm{id}$.

\emph{Equivariance.} For $g\in G$ put $\theta(g):=(g^{\dagger})^{-1}$.
By Lemma \ref{selfdualW}, $g^{\dagger}$ preserves
$(W^i_a)^{\perp}=W^i_{-a-1}$ for all $a$, hence preserves both filtrations;
and $\langle gu,v\rangle=\langle u,g^{\dagger}v\rangle$ together with the
nondegeneracy of the $\Grr\Grr$--pairing (Lemma \ref{grgrpairing}) shows
$g^{\dagger}$ induces the identity on $\Grr\Grr$; so
$\theta(g)\in G$, and $\theta$ is an automorphism of $G$ with
$\theta^2=\mathrm{id}$ commuting with $\log$ and $\exp$. Directly from the
definitions,
$$
\sigma(g\cdot s)=\theta(g)\cdot\sigma(s) .
$$

\emph{Fixed point.} Pick $s_0\in\mathcal S$ and write
$\sigma(s_0)=g\cdot s_0$, $g\in G$. Applying $\sigma$:
$s_0=\sigma^2(s_0)=\theta(g)\,g\cdot s_0$, so $\theta(g)g=1$, i.e.\
$\theta(g)=g^{-1}$. Let $h:=g^{1/2}\in G$; then
$\theta(h)=\exp(\tfrac12\log\theta(g))=\exp(-\tfrac12\log g)=h^{-1}$, and
$$
\sigma(h\cdot s_0)=\theta(h)\cdot\sigma(s_0)=h^{-1}g\cdot s_0
= h^{-1}h^2\cdot s_0 = h\cdot s_0 .
$$
Thus $s:=h\cdot s_0$ is $\sigma$--fixed, i.e.\
$V_{a,b}=V^{\sigma}_{a,b}$, which says precisely
$\langle V_{a,b},V_{c,d}\rangle=0$ unless $(c,d)=(-a,-b)$: property (i).

\emph{Property (ii).} The weight filtration $W(N_2)$ is canonically
attached to $N_2$ (Jacobson--Morozov uniqueness), hence preserved by every
operator commuting with $N_2$, in particular by $N_1$. Since also
$N_1W(N_1)_a\subset W(N_1)_{a-2}$, for any splitting of the pair one gets
$N_1(V_{a,b})\subset W(N_1)_{a-2}\cap W(N_2)_b
=\bigoplus_{a'\leq a-2,\ b'\leq b}V_{a',b'}$, and symmetrically.
\end{proof}

\subsection{Cubical Mayer--Vietoris bookkeeping}
\label{app:cubical}

\begin{lemma}[Details for Lemma \ref{cubicaldefisom}]\label{cubicalMVdetails}
Let $F$ be a field. Model $BGL(F)^+_{{\rm def},2}$ as the homotopy colimit,
over the poset of faces of the square $\square$, of the diagram
$$
D(\sigma):= BGL\big(F[t_{\sigma}]\big)^+ ,
\qquad
F[t_{\sigma}]=F,\ F[t_i],\ F[t_1,t_2]
$$
for $\sigma$ a vertex, edge, $2$--face, with structure maps the evaluations
$t_j\mapsto 0,1$. Concretely, realize it as
$$
X_2:= X_1\ \cup_{Y\times\partial\square}\ \big(Y\times\square\big),
\qquad Y:=BGL(F[t_1,t_2])^+ ,
$$
where $X_1$ (the boundary ring) is the union, over the four edges, of the
cylinders $BGL(F[t_i])^+\times[0,1]$, glued at the four vertex copies of
$BGL(F)^+$ along the evaluation maps; all gluings are along cofibrations,
so the strict colimit computes the homotopy colimit. Then each vertex
inclusion $BGL(F)^+\hookrightarrow X_2$ is a homology isomorphism (with
arbitrary constant coefficients), and an isomorphism on $\pi_1=K_1(F)$.
\end{lemma}
\begin{proof}
Write $H_{\ast}$ for homology with the fixed coefficients and
$H:=H_{\ast}(BGL(F)^+)$. All the structure maps of the diagram are homology
equivalences: the evaluations
$e_{0},e_1: BGL(F[t_i])^+\rar BGL(F)^+$ and
$BGL(F[t_1,t_2])^+\rar BGL(F[t_i])^+$ are so by Quillen's homotopy
invariance of $K$--theory for the regular rings $F$ and $F[t_i]$
\cite{Quillen} (as in \cite[\S 7.1]{IS}, homology of the plus construction
is determined by the $K$--theory space).

\emph{Step 1: two adjacent pieces.} The union of one edge cylinder with its
two vertex copies is the double mapping cylinder of
$BGL(F)^+\stackrel{e_1}{\leftarrow}BGL(F[t_i])^+
\stackrel{e_0}{\rar}BGL(F)^+$; since both maps are homology equivalences,
each vertex inclusion into it is a homology equivalence (Mayer--Vietoris
with $A,B$ the two half--cylinders, $A\cap B\simeq BGL(F[t_i])^+$).

\emph{Step 2: the boundary ring.} Decompose $X_1=A\cup B$ where $A$ is the
union of three consecutive edge cylinders (with their vertices) and $B$ the
fourth cylinder; then $A\simeq BGL(F)^+$ (Step 1, iterated twice) and
$B\simeq BGL(F)^+$, while $A\cap B$ is the disjoint union of the two end
vertex copies, $A\cap B\simeq BGL(F)^+\sqcup BGL(F)^+$. In the
Mayer--Vietoris sequence
$$
\cdots\rar H_n(A\cap B)\rar H_n(A)\oplus H_n(B)\rar H_n(X_1)
\rar H_{n-1}(A\cap B)\rar\cdots
$$
the middle map is $H\oplus H\rar H\oplus H$,
$(x,y)\mapsto(x+y,\,x+y)$ up to identifications (each component of
$A\cap B$ includes into each of $A,B$ by a homology equivalence), with
kernel and cokernel both isomorphic to $H_n$; hence
$$
H_n(X_1)\;\cong\;H_n\oplus H_{n-1}
\;=\;H_n\big(BGL(F)^+\times S^1\big),
$$
and one checks the vertex inclusion realizes the first summand.

\emph{Step 3: gluing the face.} The attaching map
$\phi: Y\times\partial\square\rar X_1$ sends $Y\times(\mbox{edge})$ to the
corresponding edge cylinder through the evaluation
$Y\rar BGL(F[t_i])^+$ (a homology equivalence) crossed with the identity of
the interval. Running the Step--2 computation for the source
$Y\times\partial\square$ (a boundary ring built from four copies of
$Y\times[0,1]$) and comparing with that for $X_1$ through the map of
Mayer--Vietoris sequences induced by $\phi$, the five lemma shows
$\phi_{\ast}$ is an isomorphism on $H_{\ast}$. Now apply Mayer--Vietoris to
$X_2=X_1\cup(Y\times\square)$ with intersection $Y\times\partial\square$:
the map
$$
H_n(Y\times\partial\square)\lrar H_n(X_1)\oplus H_n(Y\times\square)
$$
has first component the isomorphism $\phi_{\ast}$, hence is (split)
injective; so all connecting maps vanish and
$$
H_n(X_2)\cong
\frac{H_n(X_1)\oplus H_n(Y)}{H_n(Y\times\partial\square)}
\cong
\frac{(H_n\oplus H_{n-1})\oplus H_n}{H_n\oplus H_{n-1}}
\cong H_n ,
$$
where the quotient is by the graph of
$(\phi_{\ast},\ \mbox{projection})$; a diagram chase identifies the
surviving $H_n$ with the image of a vertex copy. For $\pi_1$: all spaces
are plus constructions or products of such with $S^1$--like pieces glued
along $\pi_1$--isomorphisms on connected intersections; Van Kampen gives
$\pi_1(X_2)=K_1(F)$, abelian, with the vertex inclusion inducing the
identity; hence the spaces are simple and homology isomorphisms give weak
equivalences if desired. Cohomology with arbitrary coefficients follows by
universal coefficients.
\end{proof}

\begin{lemma}[Details for Lemma \ref{bicompare1}]\label{comparedetails}
Let $(\{\eta_{\sigma}\})$ be the $2$--cubical deformation patching datum of
\S \ref{cubicalfromrho} attached to $(\rho,\Wone,\Wtwo)$ over
$F\subset\comx$, and $\nabla^{\rm def}$ the associated connection of
Lemma \ref{triangulardef}. Then
$\what{c}_p(\nabla^{\rm def})=\what{c}_p^{\rm def}(\{\eta_{\sigma}\})$ in
$H^{2p-1}(X,\comx/\Z)$, $p\geq 1$.
\end{lemma}
\begin{proof}
\emph{(1) Finite stages and naturality.} $X$ is a finite CW complex and the
datum takes values in $GL_r$. For $N\geq r$ let
$X_2^{(N)}$ denote the cubical space of Lemma \ref{cubicalMVdetails} built
from $BGL_N$ instead of $BGL$, and $X_2^{(\infty)}=X_2$ their colimit. The
datum $\{\eta_{\sigma}\}$ defines, by the cubical decomposition of $X$
(Lemma \ref{cubicalpushout}) and functoriality of homotopy colimits, a map
$\xi_{\rho}^{(N)}: X\rar X_2^{(N)}$ for every $N\geq r$, compatible with
stabilization, and $\what{c}_p^{\rm def}$ is by definition the pullback
under $\xi_{\rho}=\xi^{(\infty)}_{\rho}$ of the universal class. Both
classes to be compared are natural: for
$\what{c}_p^{\rm def}$ this is the definition, and for
$\what{c}_p(\nabla^{\rm def})$ it follows from Lemma \ref{biinv} (any two
connections compatible with the same localized collection give the same
class) together with the fact that $\nabla^{\rm def}$ restricted to each
block or collar of $X$ is the pullback, under the block's classifying map,
of the corresponding tautological flat family.

\emph{(2) Reduction to the universal space.} By (1) it suffices to prove
the equality of the two universal classes on $X_2^{(N)}$ for $N$ large,
in cohomological degree $2p-1$: precisely, on a $\mathcal{C}^{\infty}$
model of the finite stage (replace each $BGL_N(F[t_{\sigma}])^+$ by the
realization of the simplicial manifold $N_{\bullet}GL_N$ of the discrete
group, on which the tautological flat family carries the connection
$\nabla^{\rm def}$ of Lemma \ref{triangulardef} cube by cube --- the same
construction as on $X$, applied to the tautological datum). Naturality
under $\xi^{(N)}_{\rho}$ then pulls the universal identity back to $X$.
By homological stability for general linear groups over an infinite field
\cite{vdK}, $H_n(GL_N(F'))\rar H_n(GL(F'))$ is an isomorphism for
$N\gg n$, uniformly over the finitely many rings
$F'=F,F[t_i],F[t_1,t_2]$ appearing (each is a ring with many units); so in
the fixed degree $2p-1$ the Mayer--Vietoris computation of Lemma
\ref{cubicalMVdetails} holds verbatim on the finite stage $X_2^{(N)}$ for
$N$ large: a class in $H^{2p-1}(X_2^{(N)},\comx/\Z)$ is determined by its
restriction to one (hence each) vertex piece.

\emph{(3) Vertex identification.} On the vertex piece indexed by
$(i,j)\in\{0,1\}^2$, the restriction of the universal
$\what{c}^{\rm def}_p$ is the universal regulator class of the tautological
representation $\eta_{ij}$ (evaluation of the classifying map at the
vertex), while the restriction of $\what{c}_p(\nabla^{\rm def})$ is the
Cheeger--Simons class of the tautological \emph{flat} connection of
$\eta_{ij}$ (the connection $\nabla^{\rm def}$ is flat over the vertex
blocks). These agree: for flat bundles the Cheeger--Simons class coincides
with the universal secondary (regulator) class, by the comparison of
Beilinson/Esnault/Cheeger--Simons theories recalled in \cite[\S 2]{IS}.
Hence the two universal classes agree on each vertex piece, and by (2) they
are equal on $X_2^{(N)}$; by (1) their pullbacks to $X$ agree.
\end{proof}


\end{document}